\documentclass[12pt, twoside]{article}
\usepackage{amsfonts}
\usepackage{xcolor}
\usepackage{mathrsfs}
\usepackage[all,cmtip]{xy}
\usepackage{amsmath}
\usepackage{amssymb,amsthm,upref,amscd}
\usepackage{setspace}
\usepackage{enumerate}
\usepackage[titletoc]{appendix}
\usepackage{times}
\usepackage{cite}
\usepackage{tikz}
\usepackage[colorinlistoftodos]{todonotes}
\usetikzlibrary{arrows}
\numberwithin{equation}{section}

\def\titlerunning#1{\gdef\titrun{#1}}
\makeatletter
\def\author#1{\gdef\autrun{\def\and{\unskip, }#1}\gdef\@author{#1}}

\makeatother

\def\keywords#1{\par
\noindent\textbf{Keywords.} #1}

\allowdisplaybreaks

\theoremstyle{plain}
\newtheorem{Thm}{Theorem}[section]
\newtheorem{Lem}[Thm]{Lemma}

\newtheorem{Cor}[Thm]{Corollary}
\newtheorem{Prop}[Thm]{Proposition}
\newtheorem*{Thm*}{Theorem}
\newtheorem*{claim*}{Claim}
\theoremstyle{definition}

\newtheorem*{Def*}{Definition}
\newtheorem*{Cor*}{Corollary}
\newtheorem{Rem}[Thm]{Remark}

\newcommand{\equ}{equation}
\newcommand{\C}{\mathbb{C}}
\newcommand{\N}{\mathbb{N}}
\newcommand{\R}{\mathbb{R}}
\newcommand{\Z}{\mathbb{Z}}

\DeclareMathOperator{\dist}{dist}

\DeclareMathOperator{\supp}{supp}

\DeclareMathOperator{\vol}{vol}
\DeclareMathOperator{\Real}{Re}

\newcommand\eps{\varepsilon}

\let\nhatoksa=\theenumi
\let\nhatoksb=\labelenumi
\let\nhatoksc=\theenumii
\let\nhatoksd=\labelenumii
\newlength{\nhalengtha}
\newlength{\nhalengthb}
\newlength{\nhalengthc}
\newcommand{\resetenum}{
\let\theenumi=\nhatoksa
\let\labelenumi=\nhatoksb
\let\theenumii=\nhatoksc
\let\labelenumii=\nhatoksd
\setlength{\leftmargini}{\nhalengtha}
\setlength{\leftmarginii}{\nhalengthb}
\setlength{\labelwidth}{\nhalengthc}
}

\newcommand\cb{\mathcal{B}}

\newcommand\cd{\mathcal{D}}
\newcommand\ce{\mathcal{E}}

\newcommand\ch{\mathcal{H}}

\newcommand\cl{\mathcal{L}}
\newcommand\cm{\mathcal{M}}

\newcommand\rr{\mathcal{R}}

\def\mbs{\mathbb{S}}

\def\msh{\mathscr{H}}

\def\msn{\mathscr{N}}

\def\id{\text{Id}}
\def\ig{\textit{g}}

\def\ov{\overline}

\def\pa {\partial}

\def\op{\oplus}
\def\ot{\otimes}

\def\De{\Delta}

\def\al{\alpha}
\def\bt{\beta}

\def\de{\delta}
\def\Ga{\Gamma}
\def\ga{\gamma}

\def\lm{\lambda}

\def\om{\omega}

\def\sa{\sigma}

\def\vr{\varepsilon}
\def\va{\varphi}

\def\div{\hbox{div}}

\def\vol{\mathrm{vol}}

\newcommand{\inp}[2]{\left\langle#1,#2\right\rangle}

\def\clf{\C\ell}

\begin{document}

\titlerunning{Curvature effect in the spinorial Yamabe problem}

\title{Curvature effect in the spinorial Yamabe problem on product manifolds}

\author{Thomas Bartsch,\, Tian Xu\footnote{Supported by the National Science Foundation of China (NSFC 11601370)
and the Alexander von Humboldt Foundation of Germany}}

\date{}

\maketitle

\begin{abstract}

Let $(M_1,\ig^{(1)})$, $(M_2,\ig^{(2)})$ be closed Riemannian spin manifolds. We study the existence of solutions of the Spinorial Yamabe problem on the product $M_1\times M_2$ equipped with a family of metrics $\vr^{-2}\ig^{(1)}\op\ig^{(2)}$, $\vr>0$. Via variational methods and blow-up techniques, we prove the existence of solutions which depend only on the factor $M_1$, and which exhibit a spike layer as $\vr\to0$. Moreover, we locate the asymptotic position of the peak points of the solutions in terms of the curvature tensor on $(M_1,\ig^{(1)})$.
\vspace{.5cm}

\noindent{\bf MSC 2010:} Primary: 53C27; Secondary: 35B40, 35Q40, 35R01, 58E30, 58J60
\vspace{.2cm}

\keywords{Spinorial Yamabe equation; strongly indefinite functional; blow-up solutions; spike layer solutions}


\end{abstract}


%
\section{Introduction}

Let $N$ be an $n$-dimensional closed spin manifold, $n\geq2$, with Riemannian metric $\ig$ and a fixed spin structure $\sa:P_{Spin}(N)\to P_{SO}(N)$. Denoted by $\rho: Spin(n)\to End(\mbs_n)$ the spin representation, we write $\mbs(N)=P_{Spin}(N)\times_\rho \mbs_n$ for the spinor bundle over $N$ and $D_\ig^N:C^\infty(N,\mbs(N))\to C^\infty(N,\mbs(N))$ for the (Atiyah-Singer) Dirac operator.

A spinorial analogue of the Yamabe equation can be written as
\begin{\equ}\label{Dirac-c}
D_\ig^N\psi=|\psi|_\ig^{n^*-2}\psi, \quad \text{on } (N,\ig,\sa)
\end{\equ}
where $n^*:=\frac{2n}{n-1}$ --- in fact, Eq. \eqref{Dirac-c} is conformally invariant and is the Euler-Lagrange equation of a variational problem similar to Yamabe's problem (see \cite{Ammann}). In case $\psi\in C^1(N,\mbs(N))$ is a non-trivial solution to \eqref{Dirac-c}, the (generalized) conformal metric $\tilde\ig=|\psi|_\ig^{\frac4{n-1}}\ig$ induces a spinor field $\va$ on $(N,\tilde\ig,\sa)$ such that
\begin{\equ}\label{cmc-equ}
D_{\tilde\ig}^N\va=\va, \quad |\va|_{\tilde\ig}\equiv1
\end{\equ}
on $N\setminus \psi^{-1}(\{0\})$. As an important geometric application, a solution to Eq.  \eqref{cmc-equ} on a two-dimensional manifold $N$ corresponds to the existence of an isometric immersion $(\widetilde N,\ig)\to\R^3$ of the universal covering $\widetilde N$ into the Euclidean 3-space with constant mean curvature
 (see \cite{Ammann, Friedrich} and references therein for more geometric backgrounds).

A series of works of B. Ammann and his group \cite{Ammann, Ammann2003, Ammann2009, ADHH, AGHM, AHA, AHM} have provided a brief picture of how variational method is employed to the study of \eqref{Dirac-c}. From the view point of analysis, as it was pointed out in \cite{Ammann}, standard variational methods do not directly imply the existence of a solution. This is due to the  criticality of the nonlinearity in \eqref{Dirac-c}. Indeed, the exponent $n^*=\frac{2n}{n-1}$ is critical in the sense that the Sobolev embedding involved is precisely the one for which the compactness is lost. Similar to the idea in solving the Yamabe problem, it is possible to find a criterion which recovers compactness. Here the crucial observation is that a spinorial analogue of Aubin's inequality holds (see \cite{AGHM}):
\begin{\equ}\label{spinorial Aubin inequ}
\lm_{min}^+(N,[\ig],\sa):=\inf_{\tilde\ig\in[\ig]}\lm_1^+(\tilde\ig)
\text{Vol}(N,\tilde\ig)^{\frac1n}\leq \lm_{min}^+(S^n,[\ig_{S^n}],\sa_{S^n})=\frac n2 \om_n^{\frac1n}
\end{\equ}
where $[\ig]=\{f^2\ig:f\in C^1(N),\, f\geq 0, \, \supp f=N\}$ is the (generalized) conformal class of $\ig$ and, for each $\tilde\ig\in[\ig]$, $\lm_1^+(\tilde\ig)>0$ denotes the smallest positive eigenvalue of the associated Dirac operator $D_{\tilde\ig}^N$,
$(S^n,\ig_{S^n},\sa_{S^n})$ is the $n$-dimensional sphere equipped with its canonical metric $\ig_{S^n}$ and spin structure $\sa_{S^n}$, and $\om_n$ is the volume of $(S^n,\ig_{S^n})$. The quantity $\lm_{min}^+(N,[\ig],\sa)$ is known as the B\"ar-Hijazi-Lott invariant. One of the main results obtained in \cite{Ammann} shows that if inequality \eqref{spinorial Aubin inequ} is strict then the spinorial Yamabe problem \eqref{Dirac-c} has a nontrivial solution. However, the strict inequality in \eqref{spinorial Aubin inequ} is only verified for some special cases and a general result is still lacking (cf. \cite{ADHH, AHM, Ginoux}).

The purpose of this paper is to establish existence results for \eqref{Dirac-c} on products of compact spin manifolds without knowing whether the strict inequality in \eqref{spinorial Aubin inequ} holds. Moreover, we are also interested in the effect of the curvature tensors in our existence results. In particular, given closed spin manifolds $(M_1,\ig^{(1)},\sa_1)$ and $(M_2,\ig^{(2)},\sa_2)$, with fixed spin structures, let us consider a family of metrics $\ig_\vr$ on the product $N=M_1\times M_2$ defined by $\ig_\vr=\vr^{-2}\ig^{(1)}\op \ig^{(2)}$, $\vr>0$. Let $m_1$ and $m_2$ be the dimensions of $M_1$ and $M_2$ respectively, we will be interested in solutions of the spinorial Yamabe equation on the product manifold $(N,\ig_\vr)$:
\begin{\equ}\label{Dirac-product}
D_{\ig_\vr}^N\phi=|\phi|_{\ig_\vr}^{n^*-2}\phi
\end{\equ}
where $n^*=\frac{2n}{n-1}$ and $n:=\dim N=m_1+m_2$. In the sense of separation of variables, we restrict our study to spinor fields of the form $\phi=\psi\otimes\va\in C^1(N,\mbs(N))$ such that $\psi\in C^1(M_1,\mbs(M_1))$ and $\va\in C^1(M_2,\mbs(M_2))$ are spinor fields on $M_1$ and $M_2$ respectively.

In order to describe our main results, it is useful to recall some notation and definitions in differential geometry, see for instance \cite{Chavel, Jost}. For a closed Riemannian $m$-manifold $(M,\ig)$, let $\exp: TM\to M$ be the exponential map defined on the tangent bundle $TM$ of $M$. Since $M$ is closed, there exists $r>0$ such that $\exp_\xi: B_r(0)\supset \R^m\cong T_\xi M\to B_r(\xi)\subset M$ is a diffeomorphism for any $\xi\in M$. Throughout the paper, $B_r(0)$ will denote the ball in $\R^m$ centered at $0$ with radius $r$ and, for $\xi\in M$, $B_r(\xi)$ will denote the ball in $M$ centered at $\xi$ with respect to the metric $\ig$. For vector fields $X,Y,W,Z$ on $M$, the Riemannian curvature tensor $\rr$ is given by
\[
\rr(X,Y,W,Z)=\ig(\nabla_{X}\nabla_{Y}W,Z)-\ig(\nabla_{Y}\nabla_{X}W,Z)-\ig(\nabla_{[X,Y]}W,Z)
\]
where $\nabla$ is the Riemannian connection. For an orthonormal basis $\{e_1,\dots,e_m\}$ of $T_\xi M$, $\xi\in M$, the Ricci tensor $\text{Ric}: T_\xi M\times T_\xi M\to\R$ is given by the trace of $\rr$, that is
\[
\text{Ric}(X,Y)=\sum_{j=1}^m\rr(e_j,X,Y,e_j)
\]
and the scalar curvature and Gaussian curvature will be denoted by $Scal_\ig(\xi)$ and $K_\ig(\xi)$ respectively. On spin manifolds, since the tangent bundle is embedded in the bundle of Clifford algebra, vector fields have two different actions on spinors, i.e. the Clifford multiplications and the covariant derivatives. Here, to distinguish these two actions on a spinor $\psi$, we denote respectively $\pa_j\cdot_\ig\psi$ the Clifford multiplication of $\pa_j$ and $\nabla_{\pa_j}\psi$ the covariant derivative, $j=1,\dots,m$, with respect to the background metric $\ig$. We also adopt the notation $f_\vr\lesssim g_\vr$ for two $\vr$-dependent functions $f_\vr$ and $g_\vr$, when there exists a constant $C>0$ independent of $\vr$ such that $f_\vr\leq C g_\vr$.

\subsection{An illustrative Example}

We begin by describing an example when $n=3$, where $N=\Sigma\times S^1$ for a closed Riemann surface $\Sigma$ and $S^1$ the standard circle. Let $\ig$ denote the background metric on $\Sigma$ and $d\tau$ denote the standard metric on $S^1$ with total length $2\pi$, then we are concerned with the product metrics $\vr^{-2}\ig\op d\tau$, $\vr>0$.
We also equip $\Sigma$ and $S^1$ with spin structures $\sa_\Sigma$ and $\sa_{S^1}$, respectively.  In this setting, we have $n^*=3$ and Eq. \eqref{Dirac-product} reads as
\begin{\equ}\label{Dirac 3d}
D_{\ig_\vr}^N\phi=|\phi|_{\ig_\vr}\phi,
\quad 
\phi: N\to\mbs(N)
\end{\equ}
where the spinor bundle $\mbs(N)$ is identified as the tensor product
$\mbs(N)=\mbs(\Sigma)\otimes\mbs(S^1)$.
Since the associated Dirac operator on $S^1$ is simply $i\frac{d}{d\tau}$, and since we are looking for a solution of the form $\phi=\psi\otimes\va$, we can take $\va=e^{-i\lm\tau}$ to be an eigen-spinor on $S^1$ (i.e. $i\frac{d}{d\tau}\va=\lm\va$) for some eigenvalue $\lm\neq0$. Then $\phi=\psi\otimes e^{-i\lm\tau}$ is a solution to Eq. \eqref{Dirac 3d} if and only if $\psi:\Sigma\to\mbs(\Sigma)$ is a solution to the following reduced equation (see Section \ref{Variational settings-sec} for a detailed explanation)
\begin{\equ}\label{reduced equ in 3d}
\vr D_{\ig}^{\Sigma}\psi+ \lm\om_\C^{\Sigma}\cdot_{\ig}\psi=
|\psi|_{\ig}\psi, \quad \text{on } \Sigma
\end{\equ}
where $\om_\C^{\Sigma}$ is the chirality operator in the Clifford bundle $\C l(T\Sigma)$ and $``\cdot_\ig"$ denotes the Clifford multiplication.

We also introduce the following equation which corresponds to a limiting equation to problem \eqref{reduced equ in 3d} as $\vr$ goes to zero:
\begin{\equ}\label{limit equ 3d}
D_{\ig_{\R^2}}\psi+ \lm\om_\C\cdot_{\ig_{\R^2}}\psi=
|\psi|\psi, \quad \psi: \R^2\to\mbs(\R^2)\cong\C^2
\end{\equ}
where $\ig_{\R^2}$ denotes the standard Euclidean metric and $\om_\C=i \pa_1\cdot_{\ig_{\R^2}}\pa_2$ is the corresponding chirality operator with $``\cdot_{\ig_{\R^2}}"$ being the Clifford multiplication and $\pa_1=\frac{\pa}{\pa x_1}, \, \pa_2=\frac{\pa}{\pa x_2}$ are the canonical base in the tangent bundle $T\R^2$.

The blow-up profiles (the so-called concentration phenomenon) appearing in solution sequences of Eq. \eqref{reduced equ in 3d} (as $\vr\to0$) are described by rescaled solutions of the above limit equation \eqref{limit equ 3d}. As we will see in Section \ref{limit equ sec}, Eq. \eqref{limit equ 3d} has a variational structure, of strongly indefinite type. Ground state solutions (i.e. solutions with minimal energy) to Eq. \eqref{limit equ 3d} can be obtained via a standard linking arguments, moreover, these solutions decay exponentially at infinity. 

Note that  Eq. \eqref{limit equ 3d} is invariant by translation, we denote $\cb$ the set of ground state solutions of \eqref{limit equ 3d} having maximum modulus at the origin, i.e.  $|\psi(0)|=\max_{x\in\R^2}|\psi(x)|$ for $\psi\in\cb$. As explained in Lemma \ref{exponential}, $\cb$ is compact in $W^{1,q}(\R^2,\mbs(\R^2))$, $q\geq2$.
Now we are ready to state the results for $N=\Sigma\times S^1$:
\begin{Thm}\label{thm-3d}
There exists $\vr_0>0$ such that, for any $\vr\in(0,\vr_0)$, Eq. \eqref{reduced equ in 3d} has a solution $\psi_\vr\in C^1(\Sigma,\mbs(\Sigma))$. Furthermore, there is a maximum point $y_\vr\in \Sigma$ such that
\begin{itemize}
\item[$(1)$] for a constant $c>0$,
\[
|\psi_\vr(\xi)|_\ig\lesssim \exp\Big( -\frac{c}\vr\dist(\xi,y_\vr) \Big),\quad \text{for all } \xi\in \Sigma
\]
where $\dist$ is the distance induced by the metric $\ig$;

\item[$(2)$] as $\vr\to0$, up to a subsequence, the transformed spinor $z_\vr(x)\equiv\psi_\vr\circ\exp_{y_\vr}(\vr x)$ converges uniformly to a ground state solution $z_0\in\cb$ of \eqref{limit equ 3d} and $y_\vr\to y_0$ in $(\Sigma,\ig)$ such that
 \[
\Theta(y_0,z_0)=\max_{(y,\psi)\in \Sigma\times\cb}\Theta(y,\psi),
\]
where $\Theta: \Sigma\times\cb\to\R$ is a functional defined by
\begin{\equ}\label{functional theta}
\aligned
\Theta(y,\psi)&=\frac{K_\ig(y)}{36}\int_{\R^2}|\psi|^3|x|^2dx\\
 &\qquad +\frac{K_\ig(y)}{12}\Real\int_{\R^2}\big( (x_2\nabla_{\pa_1}-x_1\nabla_{\pa_2})\psi, (x_2\pa_1-x_1\pa_2)\cdot_{\ig_{\R^2}}\psi \big) dx.
\endaligned
\end{\equ}
\end{itemize}
\end{Thm}

The result presented above implies that the concentrating behavior of a solution to \eqref{Dirac 3d} is affected by the Gaussian curvature $K_\ig$ on $\Sigma$. Since $\Sigma\times\cb$ is compact,  a maximizer of $\Theta$ does exist.

\begin{Rem}\label{Remark1}
	The spin structure of Euclidean spaces is quite explicit and equation \eqref{limit equ 3d} can be rewritten in matrix notation as
	\begin{\equ}\label{limit equ 3d matrix}
		\Big(\ga_1\frac{\pa}{\pa x_1}+\ga_2\frac{\pa}{\pa x_2}\Big)\psi+\lm\ga_3\psi=|\psi|\psi
	\end{\equ}
where $\ga_k$, $k=1,2,3$, are the $2\times2$ Pauli matrices
\[
\ga_1=\begin{pmatrix}
	0&i\\
	i&0
\end{pmatrix}, \quad
\ga_2=\begin{pmatrix}
	0&1\\
	-1&0
\end{pmatrix}, \quad
\ga_3=\begin{pmatrix}
	1&0\\
	0&-1
\end{pmatrix}.
\]
	Passing to polar coordinates in $\R^2$, i.e. $(x_1,x_2)\mapsto (r,\vartheta)$, Eq. \eqref{limit equ 3d matrix} reads as
	 \[
	 \left\{
	 \aligned
	 &e^{-i\vartheta}\Big( i\frac{\pa}{\pa r} +\frac1r\frac{\pa}{\pa\vartheta} \Big)\psi_2=\sqrt{|\psi_1|^2+|\psi_2|^2}\,\psi_1-\lm\psi_1 \\[0.3em]
	 &e^{i\vartheta}\Big( i\frac{\pa}{\pa r} -\frac1r\frac{\pa}{\pa\vartheta} \Big)\psi_1=\sqrt{|\psi_1|^2+|\psi_2|^2}\,\psi_2+\lm\psi_2 
	 \endaligned
	 \right.
	 \]
	 where $\psi=\begin{pmatrix}
	 	\psi_1\\
	 	\psi_2
	 \end{pmatrix}\in\C^2$, and this suggests the following special ansatz (see \cite{CKSCL})
 \begin{\equ}\label{ansatz}
 \psi(r,\vartheta)=\begin{pmatrix}
 	v(r)e^{iS\vartheta}\\[0.3em]
 	iu(r)e^{i(S+1)\vartheta}
 \end{pmatrix}, \quad 
r>0,\ \vartheta\in[0,2\pi)
 \end{\equ}
with $u, v$ real-valued and $S\in\Z$. Plugging such ansatz into the functional $\Theta$, we find that the second term in \eqref{functional theta} vanishes, i.e. we get a simpler expression of $\Theta$ as
\[
\Theta(y,\psi)=\frac{\pi K_\ig(y)}{18}\int_0^\infty\big(u^2+v^2\big)^{\frac32}r^3dr.
\]
This would give a simplified view of
the concentration phenomenon in Theorem \ref{thm-3d}, i.e. $y_0$ must be a global maximum point of the Gaussian curvature $K_\ig$ on $(\Sigma,\ig)$.
Unfortunately, we find no evidence that ground state solutions to the limit equation \eqref{limit equ 3d} should be in the form \eqref{ansatz}. This may lead to conjecture that, up to translations and certain group actions (for instance, the multiplication by $e^{i\om}$ for $\om\in [0,2\pi]$) , the ground state solution $\psi$ to Eq. \eqref{limit equ 3d} is uniquely determined and  takes the form of \eqref{ansatz} (or may be other symmetric ansatz). We add that solutions of the form \eqref{ansatz} to the 2D nonlinear Dirac equation with Kerr-type critical nonlinearity $|\psi|^2\psi$ have been studied in \cite{Borrelli2017, Borrelli2020}. Furthermore, ground state solutions to the spinorial Yamabe equation \eqref{Dirac-c} on $\R^n$ has been recently classified in \cite{BMW}.
 
\end{Rem}

%

\subsection{Full statement of the results}

In order to develop an existence and concentration theory for Eq. \eqref{Dirac-product} on general product spaces $N=M_1\times M_2$, we first introduce explicitly  the spinor bundle $\mbs(N)$ over $N$ in terms of the factors $M_1$ and $M_2$, that is
\[
\mbs(N)=\left\{
\aligned
&(\mbs(M_1)\op\mbs(M_1))\ot \mbs(M_2) &\quad & \text{both } m_1 \text{ and } m_2 \text{ are odd}, \\
&\qquad \mbs(M_1)\ot\mbs(M_2) &\quad & \text{else}.
\endaligned \right.
\]
In this setting, there is no difficulty to understand that a spinor $\phi\in\mbs(N)$ has the form $\phi=\psi\otimes\va$ where $\va\in\mbs(M_2)$ and $\psi=\psi_1\oplus \psi_2\in \mbs(M_1)\oplus\mbs(M_1)$ if both $m_1$ and $m_2$ are odd, and $\psi\in \mbs(M_1)$ if $m_1$ or $m_2$ is even. 

Motivated by the example mentioned previously, we impose the following hypothesis on the second factor $(M_2,\ig^{(2)},\sa_2)$:
\begin{itemize}
\item[$(H)$] there is a solution $\va_\lm$ with constant length $|\va_\lm|_{\ig^{(2)}}=1$ of the Dirac equation $D_{\ig^{(2)}}^{M_2}\va=\lm\va$, for some $\lm\neq0$.
\end{itemize}
Under this hypothesis, $\phi=\psi\otimes\va_\lm$ is a solution of Eq. \eqref{Dirac-product} if and only if $\psi$ is a solution of the following reduced equation
\begin{\equ}\label{reduced equ in nd}
\vr\tilde D_{\ig^{(1)}}^{M_1}\psi+ \lm\om_\C^{M_1}\cdot_{\mbox{\tiny $\ig^{(1)}$}}\psi=
|\psi|_{\ig^{(1)}}^{n^*-2}\psi, \quad \text{on } M_1
\end{\equ}
where 
\[
\tilde D_{\ig^{(1)}}^{M_1}=\left\{
\aligned
&D_{\ig^{(1)}}^{M_1}\op -D_{\ig^{(1)}}^{M_1} &\quad & \text{both } m_1 \text{ and } m_2 \text{ are odd}, \\
&\qquad D_{\ig^{(1)}}^{M_1} &\quad & \text{if } m_1 \text{ is even},
\endaligned
\right.
\]
and
$\om_\C^{M_1}\cdot_{\mbox{\tiny $\ig^{(1)}$}}$ denotes the action of the chirality operator with respect to the metric $\ig^{(1)}$. In case $m_1$ is odd and $m_2$ is even, one may interchange $M_1$ and $M_2$ to get the above equation.

The corresponding limit equation associated to \eqref{reduced equ in nd} is the following one:
\begin{\equ}\label{limit equ nd}
\tilde D_{\ig_{\R^{m_1}}}\psi + \lm \om_\C\cdot_{\ig_{\R^{m_1}}}\psi = |\psi|^{n^*-2}\psi
\quad \text{on } \R^{m_1}
\end{\equ}
where $\ig_{\R^{m_1}}$ is the standard Euclidean metric, 
\[
\tilde D_{\ig_{\R^{m_1}}}=\left\{
\aligned
&D_{\ig_{\R^{m_1}}}\op -D_{\ig_{\R^{m_1}}} &\quad & \text{if both } m_1 \text{ and } m_2 \text{ are odd}, \\
&\qquad D_{\ig_{\R^{m_1}}} & \quad &\text{if } m_1 \text{ is even},
\endaligned
\right.
\]
and $\om_\C=i^{[\frac{m_1+1}2]}\pa_1\cdot_{\ig_{\R^{m_1}}}\cdots\cdot_{\ig_{\R^{m_1}}}\pa_{m_1}$ is the corresponding chirality operator in the Clifford algebra $\C l(T\R^{m_1})$ with $``\cdot_{\ig_{\R^{m_1}}}"$ being the Clifford multiplication; $\pa_1=\frac{\pa}{\pa x_1}, \dots, \pa_{m_1}=\frac{\pa}{\pa x_{m_1}}$ are the canonical base in the tangent bundle $T\R^{m_1}$.

Eq. \eqref{limit equ nd} can be regarded as the Euler-Lagrange equation for the functional 
\[
\cl(\psi)=\frac12\int_{\R^{m_1}}\big( \tilde D_{\ig_{\R^{m_1}}}\psi + \lm \om_\C\cdot_{\ig_{\R^{m_1}}}\psi, \psi \big)dx
-\frac1{n^*}\int_{\R^{m_1}}|\psi|^{n^*}dx
\]
defined for $\psi\in H^{1/2}(\R^{m_1},\mbs(\R^{m_1}))$. Since the spectrum of the linear differential operator $\tilde D_\lm=\tilde D_{\ig_{\R^{m_1}}}+ \lm \om_\C$ is given by
$Spec(\tilde D_\lm)=\big(-\infty,-|\lm|\big]\cup\big[|\lm|,+\infty\big)$, the above functional is strongly indefinite.
Several techniques have been introduced to handle such situations (see for instance \cite{BD2006,BR1979, BJS, Hofer1983, Szulkin-Weth:2010} and references therein). Notice that we have set $n=m_1+m_2$ and $n^*=\frac{2n}{n-1}$, we see the Sobolev
embedding
\[
H^{1/2}(\R^{m_1},\mbs(\R^{m_1}))\hookrightarrow L^{n^*}(\R^{m_1},\mbs(\R^{m_1}))
\]
is locally compact (due to $n^*<m_1^*=\frac{2m_1}{m_1-1}$). This means that Palais-Smale sequences for the functional $\cl$ possess local strong convergence in $L^{n^*}$ and thus in $H^{1/2}$. In the framework of Concentration-Compactness theory, one obtains the existence of ground state solutions to Eq. \eqref{limit equ nd} via standard variational arguments.

For ease of notation, we still denote $\cb$ the set of all ground state solutions of \eqref{limit equ nd} satisfying $|\psi(0)|=\max_{x\in\R^{m_1}}|\psi(x)|$ for $\psi\in\cb$. We know from Lemma \ref{exponential} that $\cb$ is compact in $W^{1,q}(\R^{m_1},\mbs(\R^{m_1}))$, $q\geq2$.
Then our main result reads as

\begin{Thm}\label{thm main}
Assume $\dim M_1=m_1\geq2$ and $(M_2,\ig^{(2)},\sa_2)$ satisfies hypothesis $(H)$. Let $m_1$ be even when $n=m_1+m_2$ is odd. Then there exists $\vr_0>0$ such that, for any $\vr\in(0,\vr_0)$, Eq. \eqref{reduced equ in nd} has a solution of the form $\psi_\vr\in C^1(M_1,\mbs(M_1))$. Furthermore, there is a maximum point $y_\vr\in M_1$ such that
\begin{itemize}
\item[$(1)$] for a constant $c>0$,
\[
|\psi_\vr(\xi)|_{\ig^{(1)}}
\lesssim \exp\Big( -\frac{c}{\vr}\dist^{(1)}(\xi,y_\vr) \Big),\quad \text{for all } \xi\in M_1
\]
where $\dist^{(1)}$ is the distance induced by the metric $\ig^{(1)}$;

\item[$(2)$] as $\vr\to0$, up to a subsequence, the transformed spinor $z_\vr(x)\equiv\psi_\ell\circ\exp_{y_\vr}(\vr x)$ converges uniformly to a ground state solution $z_0\in\cb$ of \eqref{limit equ nd} and $y_\vr\to y_0$ in $(M_1,\ig^{(1)})$ such that
\[
\Theta(y_0,z_0)=\max_{(y,\psi)\in M_1\times\cb}\Theta(y,\psi),
\]
where $\Theta: M_1\times\cb\to\R$ is a functional defined by
\[
\aligned
\Theta(y,\psi)&=\frac1{12n}\int_{\R^{m_1}}\text{Ric}_y(x,x)|\psi|^{\frac{2n}{n-1}}dx \\
&\qquad +\frac1{12}\sum_{j,k}\Real\int_{\R^{m_1}}\rr_y(e_j,x,x,e_k)
(\nabla_{\pa_k}\psi,\pa_j\cdot_{\ig_{\R^{m_1}}}\psi)dx.
\endaligned
\]
\end{itemize}
\end{Thm}

The case $m_1=m_2=1$, which corresponds to $N=S^1\times S^1$, is rather simple and does not reflect the effect of curvature since the problem is reduced to an ordinary differential equation on the first circle (see \cite{SX2020}).

\begin{Rem}
\begin{itemize}
\item[a)] The hypothesis $(H)$ is rather harmless. This is satisfied by a large class of manifolds including the circle $S^1$, the $m$-spheres and many conformally flat manifolds (see for instance \cite{AHM, ADHH}).

\item[b)] Theorem \ref{thm main} does not treat directly the case of $m_1$ odd and $m_2$ even. As the statements show, the concentration phenomenon will be obtained on the even dimensional factor. This is mainly due to the specific spin-representation (see \eqref{product spinor representation} below) when $n=m_1+m_2$ is odd.

\item[c)] Analogously to Remark \ref{Remark1}, if we have additionally a symmetric characterization of ground state solutions to Eq. \eqref{limit equ nd} then the functional $\Theta$ can be intensively simplified. In higher dimensions, it is still not very clear how to give a general symmetric characterization of the solutions (while the Laplacian commutes with rotations, it is not the case for Dirac operator). However in 3D Dirac equations, it is known that there is one candidate symmetric ansatz (see in \cite{Sere2, Wakano}).
\end{itemize}
\end{Rem}

Finally we would like to compare problem \eqref{reduced equ in nd} with its counterpart of  elliptic type:
\begin{\equ}\label{single}
-\vr^2\De_g u + u = u^{p-1}, \quad u>0
\end{\equ}
on a smooth closed Riemannian manifold $(M,\ig)$, with $\dim M=m\geq3$ and $p\in(2,\frac{2m}{m-2})$.
An interesting observation is that our results about \eqref{reduced equ in nd} depend on both curvature tensors and the ground state solutions of the limit equations. In fact, the point $y_0\in M_1$ in our result locates the blow-up (or concentration) while the ground state solution $z_0\in\cb$ gives the profile of the blowing-up bubbles. Unlike the well-known results about \eqref{single} in \cite{BP, DMP, MP} etc. where only scalar curvature enters. The functional $\Theta$ in our theorems appear to be complicated and  mysterious, in particular, the second term in $\Theta$ is not very clear to us. This is because there is very little information available for the ground state solutions of the strongly indefinite problems \eqref{limit equ 3d} and \eqref{limit equ nd}. Since the limit equation of \eqref{single} on the Euclidean space $\R^m$ is explicitly understood, for which there exists a unique positive solution (up to translations) and is radially symmetric, the results for positive solutions of \eqref{single} only depends on geometric quantities. It would be very interesting to characterize those ground state solutions for Dirac equations \eqref{limit equ 3d} and \eqref{limit equ nd}, and then one may have a better understanding for the functional $\Theta$.

\subsection{Outline of the paper}

The proof of our results will be carried out in several steps. First, in Section \ref{sec: bundle setting}, we recall some preliminaries and also fix our notation. In particular, we will formally provide the spinor bundles and Dirac operators on product spaces. In Section \ref{Variational settings-sec}, we explicitly introduce Eq. \eqref{reduced equ in nd} as the reduced equation of Spinorial Yamabe equation \eqref{Dirac-product}, and set up the associated variational framework. The existence result is standard since the nonlinearity in Eq. \eqref{reduced equ in nd} has subcritical growth (in the sense of Sobolev embedding). In this case the concentration phenomenon manifests itself in the difficulty of locating the behavior of the solutions when the parameter $\vr$ is small.
Here, the key point lies in Corollary \ref{key corollary} which provides a refined upper bound estimate for the critical levels in our variational framework. In fact, Corollary \ref{key corollary} makes it possible to compute a asymptotic expansion of the critical levels in terms of $\vr$. Section \ref{Proof of the main results sec} is devoted to give the complete proof of our main results. In this section, we first collect basic properties of the ground state solutions of the limit equation \eqref{limit equ nd} since they perform as bubbles in the concentration phenomenon. The analysis of the concentration phenomenon is quite delicate and it requires a careful asymptotic expansion of the critical levels. With the help of  a well adopted spinor bundle trivialization (the so-called Bourguignon-Gauduchon trivialization) and our Corollary \ref{key corollary}, we establish the critical level expansion in terms of $\vr$ in which the effect of curvature tensors enters.


\section{Spinor bundles and Dirac operators on product spaces}\label{sec: bundle setting}
In this section, we collect some basic notations from spin geometry. Instructional material can be found in \cite[Chapter I. 5 and II. 7]{Lawson}.

\subsection{Algebraic preliminaries}
Let us denote by $\{e_1,\dots, e_m\}$ the canonical basis of an oriented Euclidean space $V$ and by $\clf(V)$ the complex Clifford algebra of $V$ with its multiplication being denoted by ``$\cdot$". In case the dimension $m$ of $V$ is even, i.e. $m=2k$, the Clifford algebra is isomorphic to the algebra $\cm(2^k;\C)$ of all complex $2^k\times2^k$ matrices. Hence $\clf(V)$ has precisely one irreducible module, the spinor module $\mbs_{2k}$ with $\dim_\C\mbs_{2k}=2^k$. When restricting this representation to the even subalgebra $\clf^0(V)$, the module $\mbs_{2k}$ splits into two irreducible unitary representations $\mbs_{2k}=\mbs_{2k}^+\op\mbs_{2k}^-$, given by the eigensubspaces of the endomorphism $\om_\C^V:=i^k e_1\cdots e_m$ to the eigenvalues $\pm1$. In the context, we will call $\om_\C^V$ the ``chirality operator" or the ``complex volume element".

In case $m$ is odd, that is $m=2k+1$, the Clifford algebra $\clf(V)$ is isomorphic to $\cm(2^k;\C)\op\cm(2^k;\C)$. And thus, we obtain two $2^k$-dimensional irreducible spinor modules $\mbs^0_{2k+1}$ and $\mbs^1_{2k+1}$ if we project the Clifford multiplication onto the first and second component respectively. Similar to the splitting in even dimensions, the two modules $\mbs^0_{2k+1}$ and $\mbs^1_{2k+1}$ can be distinguished by the chirality operator $\om_\C^V:=i^{k+1} e_1\cdots e_m$ in the sense that on $\mbs^j_{2k+1}$ it acts as $(-1)^j$, $j=0,1$. It will cause no confusion if we simply identify $\mbs^0_{2k+1}$ and $\mbs^1_{2k+1}$ as the same vector space, that is $\mbs_{2k+1}=\mbs^0_{2k+1}=\mbs^1_{2k+1}$, and equip them with Clifford multiplications of opposite signs.

Let $V$ and $W$ be two oriented Euclidean spaces with $\dim V=m_1$ and $\dim W=m_2$. We denote $\clf(V)$ and $\clf(W)$ the associated Clifford algebras of $V$ and $W$ respectively. By an abuse of notation, we use the same symbol ``$\cdot$" for the Clifford multiplication in $\clf(V)$, $\clf(W)$ and in their representations.
As it is well known, the Clifford algebra of the sum of two vector spaces is the
$\Z_2$-graded tensor product of the Clifford algebras of the two summands,
that is $\clf(V\op W)=\clf(V)\widehat\ot\clf(W)$ (see \cite{Lawson}).
Therefore, we can construct the spinor module of $V\op W$ from those of
$V$ and $W$ as
\begin{\equ}\label{spinor modules}
\mbs_{m_1+m_2}=\left\{
\aligned
&(\mbs_{m_1}\op\mbs_{m_1})\ot \mbs_{m_2} &\quad & \text{if both } m_1 \text{ and } m_2 \text{ are odd}, \\
&\qquad \mbs_{m_1}\ot\mbs_{m_2} &\quad & \text{if }  m_1 \text{ is even}.
\endaligned \right.
\end{\equ}
Here, as mentioned before, we simply excluded the case where $m_1$ is odd and $m_2$ is even because $V$ and $W$ can be interchanged. As for the representation of Clifford multiplications on $\mbs_{m_1+m_2}$, let $\xi\in V$, $\zeta\in W$, $\va\in\mbs_{m_2}$
and $\psi=\psi_1\op\psi_2\in \mbs_{m_1}\op\mbs_{m_1}$ if both $m_1$ and $m_2$ are odd, and $\psi\in\mbs_{m_1}$ if $m_1$ is even. We set
\begin{\equ}\label{clifford multiplication}
(\xi\op\zeta)\cdot (\psi\ot\va)=(\xi\cdot\psi)\ot\va+(\om_\C^V\cdot\psi)\ot (\zeta\cdot\va),
\end{\equ}
where, in case both $m_1$ and $m_2$ odd, $\xi\cdot\psi=(\xi\cdot\psi_1)\op(-\xi\cdot\psi_2)$
and $\om_\C^V\cdot\psi=i(\psi_2\op-\psi_1)$. With this notation, one easily checks that
\[
(\xi\op\zeta)\cdot(\xi\op\zeta)\cdot (\psi\ot\va)=-|\xi\op\zeta|^2(\psi\ot\va).
\]
Thus $\mbs_{m_1+m_2}$ is a nontrivial $\clf(V\op W)$-module of (complex) dimension
$2^{[\frac{m_1+m_2}{2}]}$. Moreover, in case $m_1+m_2$ is even, the splitting of
$\mbs_{m_1+m_2}$ into half-spinor modules is given by
\[
\mbs_{m_1+m_2}^+=\big\{ (\psi\op\psi)\ot\va:\, \psi\in\mbs_{m_1},\ \va\in\mbs_{m_2}\big\},
\]
\[
\mbs_{m_1+m_2}^-=\big\{ (\psi\op-\psi)\ot\va:\, \psi\in\mbs_{m_1},\ \va\in\mbs_{m_2}\big\}
\]
for both $m_1$ and $m_2$ odd and
\[
\mbs_{m_1+m_2}^+=(\mbs_{m_1}^+\ot\mbs_{m_2}^+)\op(\mbs_{m_1}^-\ot\mbs_{m_2}^-),
\]
\[
\mbs_{m_1+m_2}^-=(\mbs_{m_1}^+\ot\mbs_{m_2}^-)\op(\mbs_{m_1}^-\ot\mbs_{m_2}^+)
\]
for both $m_1$ and $m_2$ even.

Next, let us turn to the manifold setting.
Let $(M_1,\ig^{(1)})$ and $(M_2,\ig^{(2)})$ be two oriented Riemannian manifolds of dimensions $m_1$ and $m_2$, respectively. We henceforth suppose that both manifolds are
equipped with a fixed spin structure (for details about spin structures, we refer to
\cite{Friedrich, Lawson} or to the well written self-contained introduction \cite{Hij99}).
This induces a unique spin structure on the
Riemannian product $(N=M_1\times M_2,\ig=\ig^{(1)}\op\ig^{(2)})$.
Indeed, let $\pi_{M_1}$ and $\pi_{M_2}$ denote the projections
on $M_1$ and $M_2$, the tangent bundle of $N$ can be decomposed as
\[
TN=\pi_{M_1}^*TM_1\op \pi_{M_2}^*TM_2.
\]
For simplicity, we omit the projections and write $TN=TM_1\op TM_2$.
And such splitting is orthogonal with respect to $\ig$. Hence
the frame bundle of $N$ can be reduced to a $SO(m_1)\times SO(m_2)$-principal
bundle, and this is isomorphic to the product of the frame bundles over $M_1$ and $M_2$.

\subsection{The Dirac operator}
Fix the spin structures $\sa_{M_1}$ and $\sa_{M_2}$,
let us consider the Clifford bundles (with Clifford multiplications)
$(\C l(TM_1),\cdot_{\mbox{\tiny $\ig^{(1)}$}})$, $(\C l(TM_2),\cdot_{\mbox{\tiny $\ig^{(2)}$}})$ and spinor bundles
$\mbs(M_1)$, $\mbs(M_2)$ over $M_1$ and $M_2$ respectively.
From the previous considerations in the algebraic settings, we know for the spinor
bundles that
\[
\mbs(N)=\left\{
\aligned
&(\mbs(M_1)\op\mbs(M_1))\ot \mbs(M_2) &\quad & \text{if both } m_1 \text{ and } m_2 \text{ are odd}, \\
&\qquad \mbs(M_1)\ot\mbs(M_2) &\quad & \text{if } m_1 \text{ is even}.
\endaligned \right.
\]
For $X\in TM_1$, $Y\in TM_2$, $\va\in\Ga(\mbs(M_2))$ and
$\psi=\psi_1\op\psi_2\in \Ga(\mbs(M_1)\op\mbs(M_1))$ for both
$m_1$ and $m_2$ odd and $\psi\in\Ga(\mbs(M_1))$ for $m_1$ even, we have
\begin{\equ}\label{product spinor representation}
(X\op Y)\cdot_{\mbox{\tiny $\ig$}} (\psi\ot\va)=(X\cdot_{\mbox{\tiny $\ig^{(1)}$}}\psi)\ot\va+(\om_\C^{M_1}\cdot_{\mbox{\tiny $\ig^{(1)}$}}\psi)\ot(Y\cdot_{\mbox{\tiny $\ig^{(2)}$}}\va)
\end{\equ}
where in case $m_1$ and $m_2$ odd we set $X\cdot_{\mbox{\tiny $\ig^{(1)}$}}\psi=(X\cdot_{\mbox{\tiny $\ig^{(1)}$}}\psi_1)\op(-X\cdot_{\mbox{\tiny $\ig^{(1)}$}}\psi_2)$ and $\om_\C^{M_1}\cdot_{\mbox{\tiny $\ig^{(1)}$}}\psi=i(\psi_2\op-\psi_1)$.

Let $\nabla^{\mbs(M_1)}$ and $\nabla^{\mbs(M_2)}$ be the (lifted)
Levi-Civita connections on $\mbs(M_1)$ and $\mbs(M_2)$. By
\[
\nabla^{\mbs(M_1)\ot\mbs(M_2)}=\nabla^{\mbs(M_1)}\ot\id_{\mbs(M_2)}
+\id_{\mbs(M_1)}\ot \nabla^{\mbs(M_2)}
\]
we mean the tensor product connection on $\mbs(M_1)\ot\mbs(M_2)$.
If we take $\{X_1,\dots, X_{m_1}\}$ a locally positively oriented orthonormal
frame of $(M_1,\ig^{(1)})$, then the Dirac operator on $M_1$ is (locally) defined by
$D_{\ig^{(1)}}^{M_1}=\sum_{j=1}^{m_1}X_j\cdot_{\mbox{\tiny $\ig^{(1)}$}}\nabla^{\mbs(M_1)}_{X_j}$. Similarly,
if we take $\{Y_1,\dots,Y_{m_2}\}$ a locally positively oriented
orthonormal frame of $(M_2,\ig^{(2)})$, we have
$D_{\ig^{(2)}}^{M_2}=\sum_{j=1}^{m_2}Y_j\cdot_{\mbox{\tiny $\ig^{(2)}$}}\nabla^{\mbs(M_2)}_{Y_j}$.
Evidently, in the product setting,
$\{X_1\op0,\dots, X_{m_1}\op0,0\op Y_1,\dots,0\op Y_{m_2}\}$
 is a local section of the frame bundle of $N$.
Hence formula \eqref{product spinor representation} yields
\[
\aligned
D_{\ig}^N&:=\sum_{j=1}^{m_1}(X_j\op0)\cdot_{\mbox{\tiny $\ig^{(1)}$}}
\nabla^{\mbs(M_1)\ot\mbs(M_2)}_{X_j\op0}+\sum_{j=1}^{m_2}
(0\op Y_j)\cdot_{\mbox{\tiny $\ig^{(2)}$}}
\nabla^{\mbs(M_1)\ot\mbs(M_2)}_{0\op Y_j}  \\
&= \tilde D_{\ig^{(1)}}^{M_1} \ot \id_{\mbs(M_2)}+ (\om_\C^{M_1}\cdot_{\mbox{\tiny $\ig^{(1)}$}}\id_{\mbs(M_1)})\ot D_{\ig^{(2)}}^{M_2}
\endaligned
\]
which defines the Dirac operator on $N=M_1\times M_2$, where
$\tilde D_{\ig^{(1)}}^{M_1}= D_{\ig^{(1)}}^{M_1}\op -D_{\ig^{(1)}}^{M_1}$ if both $m_1$ and $m_2$ are odd
and $\tilde D_{\ig^{(1)}}^{M_1}=D_{\ig^{(1)}}^{M_1}$ if $m_1$ is even.

For the case $m_1+m_2$ even, we have the decomposition
$\mbs(N)=\mbs(N)^+\op\mbs(N)^-$ and, moreover, when restrict $D_{\ig}^N$
on those half-spinor spaces we get
$D_{\ig}^N:\Ga(\mbs(N)^\pm)\to\Ga(\mbs(N)^\mp)$.

\section{Variational setting}\label{Variational settings-sec}
In what follows, we always consider the case $N=M_1\times M_2$, $m_1=\dim M_1\geq2$ and $m_2=\dim M_2\geq1$.
In order to give unified expressions in odd and even cases,
we will write simply $\mbs(N)=\tilde\mbs(M_1)\ot\mbs(M_2)$ with
\[
\tilde\mbs(M_1)=\left\{
\aligned
&\mbs(M_1)\op\mbs(M_1) &\quad & \text{if } m_1 \text{ is odd}, \\
&\qquad \mbs(M_1) &\quad & \text{if } m_1 \text{ is even}.
\endaligned \right.
\]
and denote $\psi\ot\va$ for a spinor field in $\mbs(N)$
when no confusion can arise.

Let us consider the warped metric $\ig_\vr:=\vr^{-2}\ig^{(1)}\op\ig^{(2)}$, where $\vr>0$ is a parameter.
According to the discussions in the previous section, we know for the Dirac operators that
\[
D_{\ig_\vr}^N= \tilde D_{\vr^{-2}\ig^{(1)}}^{M_1} \ot \id_{\mbs(M_2)}+
(\ov{\om_\C^{M_1}}\cdot_{\mbox{\tiny $\vr^{-2}\ig^{(1)}$}}\id_{\tilde \mbs(M_1)})\ot D_{\ig^{(2)}}^{M_2}
\]
where $\ov{\om_\C^{M_1}}$ denotes the chirality operator and "$\cdot_{\mbox{\tiny $\vr^{-2}\ig^{(1)}$}}$" denotes the Clifford multiplication on $M_1$ associated to the conformal metric $\vr^{-2}\ig^{(1)}$ respectively.

Turning to the nonlinear problems, let us denote $|\cdot|_{\vr^{-2}\ig^{(1)}}$ and $|\cdot|_{\ig^{(2)}}$ the natural hermitian metrics on $\mbs(M_1)$ and $\mbs(M_2)$ respectively and $|\cdot|_{\ig_\vr}$ the induced metric on $\mbs(N)$. Recall the notation $n=m_1+m_2$ and $n^*=\frac{2n}{n-1}$, we can expand the spinorial Yamabe equation
\[
D_{\ig_\vr}^N\phi=|\phi|_{\ig_\vr}^{n^*-2}\phi, \quad \phi=\bar\psi\ot\va\in\mbs(N)
\]
into
\begin{\equ}\label{equ0}
(\tilde D_{\vr^{-2}\ig^{(1)}}^{M_1}\bar\psi )\ot \va+
(\ov{\om_\C^{M_1}}\cdot_{\mbox{\tiny $\vr^{-2}\ig^{(1)}$}}\bar\psi)\ot (D_{\ig^{(2)}}^{M_2}\va)=
\big( |\bar\psi|_{\vr^{-2}\ig^{(1)}}|\va|_{\ig^{(2)}} \big)^{n^*-2}\bar\psi\ot\va.
\end{\equ}

We will now show how the assumption $(H)$ on $(M_2,\ig^{(2)},\sa_{M_2})$ enters. In fact, if $M_2$
possesses a nontrivial eigenspinor $\va_{M_2}$ of constant length for some $\lm\neq0$, then
by substituting $\bar\psi\ot\va_{M_2}$ into \eqref{equ0} we get the equivalent problem
\begin{\equ}\label{equ1}
\tilde D_{\vr^{-2}\ig^{(1)}}^{M_1}\bar\psi + \lm\ov{\om_\C^{M_1}}\cdot_{\mbox{\tiny $\vr^{-2}\ig^{(1)}$}}\bar\psi=
\big(|\bar\psi|_{\vr^{-2}\ig^{(1)}}\big)^{n^*-2}\bar\psi
\end{\equ}
which is sitting on $M_1$.
Here, we adopt the convention that $\lm>0$ since (up to a change
of orientation on $M_1$) the proof for $\lm<0$ is exactly the same.

Notice that the Dirac operator behaves very nicely under conformal changes (cf. \cite{Hij86, Hit74}):  Let $\ig_0$ and $\ig=e^{2u}\ig_0$ be two conformal metrics on
a Riemannian spin $m$-manifold $M$, then  there exists an
isomorphism of vector bundles $\iota:\, \mbs(M,\ig_0)\to
\mbs(M,\ig)$ which is a fiberwise isometry such that
\[
D_\ig^M\big( \iota(\psi) \big)=\iota\big( e^{-\frac{m+1}2 u}D_{\ig_0}^M
\big( e^{\frac{m-1}2 u}\psi \big)\big).
\]
Thus Eq. \eqref{equ1} is conformally equivalent to
\begin{\equ}\label{equ2}
\vr\tilde D_{\ig^{(1)}}^{M_1}\psi+ \lm\om_\C^{M_1}\cdot_{\mbox{\tiny $\ig^{(1)}$}}\psi=
|\psi|_{\ig^{(1)}}^{n^*-2}\psi \quad \text{on } M_1
\end{\equ}
where $\om_\C^{M_1}\cdot_{\mbox{\tiny $\ig^{(1)}$}}$ denotes the action of the chirality operator with respect to the metric $\ig^{(1)}$.

\subsection{The configuration space and the Lagrangian}

Plainly, our goal is reduced to find solutions of \eqref{equ2} for varying $\vr>0$. Notice that Eq. \eqref{equ2} is well-defined on $M_1$, and $n^*=\frac{2n}{n-1}<\frac{2m_1}{m_1-1}=m_1^*$. It is not necessary to carry the super- and sub-scripts in $(M_1,\ig^{(1)})$, $\tilde D_{\ig^{(1)}}^{M_1}$ and $\om_\C^{M_1}$ during the proofs, hence in order to simplify the notation, we consider the following generalized problem
\begin{\equ}\label{equ3}
\vr\tilde D_\ig\psi+ a \om_\C\cdot_\ig\psi=f(|\psi|_g)\psi, \quad
\psi:M\to\tilde\mbs(M)
\end{\equ}
on a closed spin $m$-manifold $(M,\ig,\sa)$, where $a>0$ is a constant and $f:[0,\infty)\to[0,\infty)$ is the nonlinearity. Clearly, $f(s)=s^{n^*-2}$ is our primary concern.
Unless otherwise stated, we will occasionally drop the subscript of $|\cdot|_{\ig}$ on $\tilde\mbs(M)$ for notational convenience.

In the sequel, for $q>1$, let us denote $L^q:=L^q(M,\tilde\mbs(M))$
with the norm
$|\cdot|_q^q:=\int_M|\cdot|^q d\vol_\ig$. In particular, for $q=2$, we have $L^2$ is a Hilbert space with inner product
$(\cdot,\cdot)_2=\Real\int_M(\cdot,\cdot)d\vol_\ig$.

For each fixed $\vr>0$, let $A_\vr:=\vr\tilde D_\ig + a \om_\C\cdot_\ig$ denote the self-adjoint operator on $L^2$ with domain $\cd(A_\vr)=H^1\equiv H^1(M,\tilde\mbs(M))$. It is already known that, on a closed spin manifold $(M,\ig)$, the spectrum
$Spec(A_\vr)\subset(-\infty,-a]\cup[a,+\infty)$ is symmetric about the origin and consists of an unbounded discrete sequence of eigenvalues (with finite multiplicity for each eigenvalue), see \cite{SX2020} for details.

Thus, from the classical spectral theory of elliptic self-adjoint operators, we may choose a complete orthonormal basis $\psi_{\pm1}^\vr,\psi_{\pm2}^\vr,\dots$ of $L^2$ consisting of the eigenspinors of $A_\vr$, i.e.
$A_\vr\psi_{\pm k}^\vr=\lm_{\pm k}(\vr)\psi_{\pm k}^\vr$ and the spectrum $Spec(A_\vr)$ will be denoted as
\[
\cdots\leq \lm_{-2}(\vr)\leq \lm_{-1}(\vr)<0< \lm_1(\vr)\leq\lm_2(\vr)\leq\cdots,
\]
where each eigenvalue appears with its multiplicity.
In particular, we have $\lm_k(\vr)=-\lm_{-k}(\vr)$ and $|\lm_{\pm k}(\vr)|\to+\infty$ as $k\to\infty$.

Define the unbounded operator $|A_\vr|^s:L^2\to L^2$, $s\geq0$, by
\[
|A_\vr|^s \psi = \sum_{k=-\infty}^{\infty}|\lm_k(\vr)|^s \al_k\psi_k^\vr
\]
where $\psi=\sum_{k=-\infty}^\infty\al_k\psi_k^\vr\in L^2$. In this way, we can introduce the domain of $|A_\vr|^s$ in $L^2$ as
\[
\msh^s:=\bigg\{ \psi=\sum_{k=1}^\infty\al_k\psi_k^\vr\in L^2:\,
\sum_{k=-\infty}^\infty |\lm_k(\vr)|^{2s} |\al_k|^2<\infty \bigg\}.
\]
It is worth pointing out that $\msh^{\frac12}$ coincides with the Sobolev space of order $\frac12$, that is  $W^{\frac12,2}(M,\tilde\mbs(M))$
(see for instance \cite{Adams, Ammann}). Moreover, we can equip $\ch:=\msh^{\frac12}$ with the inner product
\begin{\equ}\label{the norm}
\inp{\psi}{\va}_\vr:=\frac1{\vr^m}\,\Real\int_M\big(|A_\vr|^{1/2}\psi,
|A_\vr|^{1/2}\va\big) d\vol_\ig
\end{\equ}
and the induced norm $\|\cdot\|_\vr$
such that $(\ch,\inp{\cdot}{\cdot}_\vr)$ becomes a Hilbert space.
Remark that, in the above notations, we have emphasized the dependence on the parameter $\vr$ because it appears in the differential operator and its spectrum. The dual space of $\ch$ will be denoted by $\ch^*=W^{-\frac12,2}(M,\tilde\mbs(M))$. Identifying $\ch$ with $\ch^*$, we will use the same notation $\inp{\cdot}{\cdot}_\vr$ to denote the norm on $\ch^*$.

Recall that we have an $(\cdot,\cdot)_2$-orthogonal decomposition
\[
L^2=L_\vr^+\op L_\vr^-, \quad \psi=\psi^++\psi^-
\]
with
\[
L_\vr^+:= \ov{\bigoplus_{k=1}^{+\infty} \ker(A_\vr-\lm_k(\vr))} \quad \text{and} \quad
L_\vr^-:= \ov{\bigoplus_{k=1}^{\infty} \ker(A_\vr-\lm_{-k}(\vr))}
\]
so that $A_\vr$ is positive definite on $L_\vr^+$ and negative definite on $L_\vr^-$. This leads to the orthogonal decomposition of $\ch$ with respect to the inner product $\inp{\cdot}{\cdot}_\vr$ as
\[
\ch=\ch_\vr^+\op \ch_\vr^-, \quad \ch_\vr^\pm=\ch\cap L_\vr^\pm.
\]

On the Banach space $L^q$, $q>1$, we introduce the new norm
\[
|\psi|_{q,\vr}=\bigg( \frac1{\vr^m}\int_M|\psi|^q d\vol_\ig \bigg)^{\frac1q} \quad \text{for each } \vr>0.
\]
Then, recall $m^*=\frac{2m}{m-1}$, we have (cf. \cite{SX2020})

\begin{Lem}\label{embeding lemma}
If $\vr>0$ is small, then for any $q\in[2,m^*]$ the embedding
$\id_\ch:(\ch,\|\cdot\|_\vr)\hookrightarrow (L^q, |\cdot|_{q,\vr})$
is bounded independent of $\vr$, that is, there exists $c_q>0$ such that
\[
|\psi|_{q,\vr}\leq c_q \|\psi\|_\vr \quad \text{for all } \psi\in\ch \text{ and all $\vr>0$ small.}
\]
In particular, the embedding is compact for $q\in[2,m^*)$.
\end{Lem}

\begin{Rem}
The proof of this lemma is a combination of the Lichnerowicz formula
\[
(\tilde D_\ig)^2=\nabla^*\nabla+\frac14 Scal_\ig,
\]
where $Scal_\ig$ is the scalar curvature of $(M,\ig)$, and an application of the Calder\'on-Lions interpolation theorem (see \cite{RS}) between $\msh^1$ and $L^2$.
In fact, for $\psi\in C^\infty(M,\tilde\mbs(M))$, we have
\begin{\equ}\label{norm esti1}
\big||A_\vr|\psi\big|_2^2=\int_M \vr^2|\nabla\psi|^2
+\Big(a^2+\frac{\vr^2}4Scal_\ig\Big)|\psi|^2 d\vol_\ig.
\end{\equ}
It follows that, for large positive $\vr$, the influence of the scalar curvature
$Scal_\ig$ enters. In this situation, \eqref{norm esti1} may no longer be a norm when $Scal_\ig$ possesses certain negative parts.
And this fact will probably affect the embedding constant of $\ch=\msh^{\frac12}\hookrightarrow L^{m^*}$.
\end{Rem}

With Lemma \ref{embeding lemma}, we introduce the following conditions for the nonlinearity $f$ in Eq. \eqref{equ3} which contains the power function $f(s)=s^{p-2}$ as a special case:
\begin{itemize}
\item[$(f_1)$] $f(0)=0$, $f\in C^1(0,\infty)$ and $f'(s)>0$ for $s>0$; 

\item[$(f_2)$] there exists $p\in(2,m^*)$, $c>0$ such that $f'(s)s\leq c(1+s^{p-2})$ for $s\geq0$;

\item[$(f_3)$]  there exists $\theta>0$ such that $f(s)\leq\frac1\theta f'(s)s$ for $s>0$.

\end{itemize}
Let $F(s)$ be the primitive function of $f(s)s$, i.e. $F(s)=\int_0^sf(t)t\,dt$, it is standard to see that Eq. \eqref{equ3} is the Euler-Lagrange equation of the functional
\begin{eqnarray}\label{the functional}
\cl_\vr(\psi)&=&\frac1{\vr^m}\int_M\Big( \frac12(A_\vr\psi,\psi)-F(|\psi|)
\Big) d\vol_\ig  \nonumber\\[0.5em]
&=&\frac12\big(\|\psi^+\|_\vr^2-\|\psi^-\|_\vr^2 \big)-
\frac1{\vr^m}\int_M F(|\psi|)d\vol_\ig
\end{eqnarray}
defined on $\ch=\ch_\vr^+\op\ch_\vr^-$. And by Lemma \ref{embeding lemma}, we have
$\cl_\vr\in C^2(\ch,\R)$.

We emphasize that the relation between $F$ and $f$ is $F'(s)=f(s)s$. And by an abuse of notation, we simply write
$\psi=\psi^++\psi^-$ for the orthogonal decomposition of $\ch$ without
mentioning its dependence on $\vr$. However, one should always keep in mind that,
for different values of $\vr$, this decomposition of a spinor $\psi$ is different.

\subsection{The reduced action}

Let us begin with the compactness of the functional $\cl_\vr$.

\begin{Lem}\label{PS-condition}
For each $\vr>0$ small, $\cl_\vr$ satisfies the $(P.S.)_c$-condition
for $c\geq0$, that is,
\[ \left.
\aligned
\cl_\vr(\psi_n) \to c \  \\
\cl_\vr'(\psi_n)\to0 \
\endaligned \right\} \Rightarrow \{\psi_n\} \text{ possesses
a convergent subsequence in } \ch.
\]
Moreover, $\psi_n\to0$ if and only if $c=0$.
\end{Lem}
\begin{proof}
Since $\cl_\vr'(\psi_n)\to0$ in $\ch^*$, we have
\begin{\equ}\label{PS1}
c+o(\|\psi_n\|_\vr)=\cl_\vr(\psi_n)-\frac12\cl_\vr'(\psi_n)[\psi_n]
=\frac1{\vr^m}\int_M \frac12 f(|\psi_n|)|\psi_n|^2-F(|\psi_n|)d\vol_\ig.
\end{\equ}
We also have
\[
o(\|\psi_n\|_\vr)=\cl_\vr'(\psi_n)[\psi_n^+-\psi_n^-]=\|\psi_n\|_\vr^2-\frac1{\vr^m}
\Real\int_M f(|\psi_n|)(\psi_n,\psi_n^+-\psi_n^-)d\vol_\ig.
\]

By $(f_1)$-$(f_3)$, one checks easily that for arbitrarily small $\de>0$ there exists $c_\de>0$ such that  $f(s)s\leq \de s +c_\de (f(s)s^2)^{\frac{p-1}{p}}$ and $F(s)\leq\frac1{\theta+2}f(s)s^2$ for all $s\geq0$. From this, together with Lemma \ref{embeding lemma}, we obtain
\begin{eqnarray}\label{PS2}
\|\psi_n\|_\vr^2 &\leq& \frac1{\vr^m}\int_M f(|\psi_n|) |\psi_n|\cdot |\psi_n^+-\psi_n^-| d\vol_\ig
+o(\|\psi_n\|_\vr)    \nonumber \\
&\leq& \de|\psi_n|_{2,\vr}|\psi_n^+-\psi_n^-|_{2,\vr}+c_\de\Big(
\frac1{\vr^m}\int_M f(|\psi_n|)|\psi_n|^2 d\vol_\ig\Big)^{\frac{p-1}p} |\psi_n^+-\psi_n^-|_{p,\vr}   +o(\|\psi_n\|_\vr) \nonumber \\
&\leq& \de c_2\|\psi_n\|_\vr^2+
C_{p,\de} \big( c+o(\|\psi_n\|_\vr) \big)^{\frac{p-1}{p}}
\|\psi_n\|_\vr + o(\|\psi_n\|_\vr) .
\end{eqnarray}
By suitable choosing $\de>0$ small, it follows from \eqref{PS2} that $\{\psi_n\}$ is bounded in $\ch$ with respect to the norm $\|\cdot\|_\vr$. And due to the compact embedding of $\ch\hookrightarrow L^{p}$, one easily checks
that $\{\psi_n\}$ is compact in $\ch$.

Finally, we mention that \eqref{PS1} and \eqref{PS2} together imply:
$\psi_n\to0$ if and only if $c=0$. This completes the proof.
\end{proof}

One may see from \eqref{the functional} that the quadratic part in the functional $\cl_\vr$ is of strongly indefinite type, i.e. positive- and negative-definite on infinite dimensional subspaces of $\ch$. Hence, in order to obtain a critical point of $\cl_\vr$, it is now crucial to find a suitable min-max scheme for $\cl_\vr$. Fortunately, due to our requests on the nonlinearity $f$ (see conditions $(f_1)$-$(f_3)$), we have a very good geometric behavior of $\cl_\vr$ in the following sense:
\begin{itemize}
	\item[$(i)$] Since the function $F$ is non-negative, for each fixed $u\in\ch_\vr^+$,  the functional
	\[
	\cl_\vr(u+\cdot): \ch_\vr^-\to\R, \quad w\mapsto \cl_\vr(u+w)
	\]
	is anti-coercive (i.e. $\cl_\vr(u+w)\to-\infty$ as $\|w\|_\vr\to\infty$).
	
	\item[$(ii)$] Since $F''(s)=f'(s)s+f(s)>0$ for $s>0$, the quadratic form $\cl_\vr''(u+w)[\cdot,\cdot]$ is negative definite on $\ch_\vr^-$, in other words, the above functional $\cl_\vr(u+\cdot)$ is strictly concave on $\ch_\vr^-$.
	
	\item[$(iii)$] Since $(f_3)$ implies that $f(s)\geq c s^{\theta}$ for some constant $c>0$ and all $s\geq1$, the function $F$ has super-quadratic growth at infinity, i.e. $F(s)\geq \frac{c}{2+\theta}s^{2+\theta}$ as $s\to\infty$. And hence, for each fixed $u\in\ch_\vr^+\setminus\{0\}$,
	\begin{\equ}\label{eeee1}
	\cl_\vr(tu+w)\to-\infty \quad \text{ as } |t|+\|w\|_\vr\to\infty.
	\end{\equ}
\end{itemize}
Combining all the above three properties, for each $u\in\ch_\vr^+\setminus\{0\}$, we are able to maximize the functional $\cl_\vr$ on $\R^+u\op\ch_\vr^-$, where $\R^+=(0,\infty)$. We remark that $u$ is varying in the space $\ch_\vr^+\setminus\{0\}$, so we can restrict ourselves to the choice $u\in S_\vr^+:=\{u\in \ch_\vr^+:\,\|u\|_\vr=1\}$ without changing the maxima of $\cl_\vr$ on $\R^+u\op\ch_\vr^-$. By $(f_1)$-$(f_2)$, one deduces that $F(s)\leq \frac\de2 s^2+ C_\de s^p$ for arbitrarily small $\de>0$ and hence (by Lemma \ref{embeding lemma})
\begin{\equ}\label{eeee2}
\max_{\R^+u\op\ch_\vr^-}\cl_\vr\geq \max_{t>0}\cl_\vr(tu)\geq \max_{t>0}\Big(\frac{1-\de}2 t^2-C_\de t^p\Big)=\tau_0>0.
\end{\equ}
Therefore, we have found a candidate min-max scheme for $\cl_\vr$: we first maximize the functional $\cl_\vr$ on $\R^+u\op\ch_\vr^-$ and then minimize with respect to $u\in\ch_\vr^+\setminus\{0\}$.

\medskip

By summarizing the above observations, we have the following basic conclusion.
\begin{Prop}\label{reduction1}
For each $\vr>0$ small the following holds.
\begin{itemize}
\item[$(1)$] There exists $\chi_\vr\in C^1(\ch_\vr^+,\ch_\vr^-)$ such that for $u\in \ch_\vr^+$
\[
w\in\ch_\vr^-, \  w\neq \chi_\vr(u) \Rightarrow
\cl_\vr(u+w)<\cl_\vr(u+\chi_\vr(u));
\]
that is $\chi_\vr(u)$ is the unique maximizer of the functional $w\mapsto\cl_\vr(u+w)$ on $\ch_\vr^-$. Moreover, for $u\in \ch_\vr^+$
\[
\|\chi_\vr(u)\|_\vr^2\leq \frac2{\vr^m}\int_M F(|u|)d\vol_\ig
\]
and $\cl_\vr'(u+\chi_\vr(u))[w]\equiv 0$ for all $w\in\ch_\vr^-$;

\item[$(2)$] If $\{u_n\}$ is a $(P.S.)$-sequence for the reduced functional
\[
I_\vr: \ch_\vr^+\to \R, \quad I_\vr(u)=\cl_\vr(u+\chi_\vr(u)),
\]
then $\{u_n+\chi_\vr(u_n)\}$ is a
$(P.S.)$-sequence for $\cl_\vr$; 

\item[$(3)$] There exists $u_\vr\in\ch_\vr^+$ such that $I_\vr'(u_\vr)=0$ and
\[
I_\vr(u_\vr)=\mu_\vr:=\inf_{\ga\in\Ga_\vr}\max_{t\in[0,1]}I_\vr(\ga(t))>0,
\]
where $\Ga_\vr=\big\{\ga\in C([0,1],\ch_\vr^+):\, \ga(0)=0,\, I_\vr(\ga(1))<0\big\}$.
In particular, $\bar\psi_\vr=u_\vr+\chi_\vr(u_\vr)$ is a non-trivial critical point of $\cl_\vr$.
\end{itemize}
\end{Prop}
\begin{proof}
Similar assertions can be found in \cite[Section 2]{BJS} where a certain abstract theory has been set up. We only mention here that the existence and uniqueness of the maximizer $\chi_\vr(u)$ in assertion $(1)$ follows directly from the anti-coerciveness and strict concavity of the functional $\cl_\vr(u+\cdot)$ on $\ch_\vr^-$. The $C^1$-smoothness of $\chi_\vr$ is a consequence of the Implicit Function Theorem. 
\end{proof}

We next provide an upper bound estimate of the critical level $\mu_\vr$ obtained in the above proposition.

\begin{Lem}\label{reduction2}
For every $u\in \ch_\vr^+\setminus\{0\}$, the map $I_{\vr,u}:\R\to\R$,
$I_{\vr,u}(t)=I_\vr(tu)$, is of class $C^2$ and satisfies $I_{\vr,u}(0)=I_{\vr,u}'(0)=0$ and $I_{\vr,u}''(0)>0$. Moreover, there holds
\[
I_{\vr,u}'(t)=0,\ t>0 \quad \Longrightarrow \quad I_{\vr,u}''(t)<0.
\]
\end{Lem}
\begin{proof}
First of all, we can use the fact that $I_{\vr,u}'(t)=\cl_\vr'(tu+\chi_\vr(tu))[u]$ to see that $I_{\vr,u}$ is of class $C^2$. By $(f_1)-(f_3)$
there holds
\[
	I_\vr(tu)\geq\cl_\vr(tu)
	\geq\frac{(1-\de)t^2}2\|u\|_\vr^2- C_{p,\de}\, t^{p}  \|u\|_\vr^{p}
	\quad \text{$\forall u\in\ch_\vr^+\setminus\{0\}$ and $t>0$,}
\]
for any fixed $\de>0$ small. Hence $I_\vr(0)=0$ is a strict local minimum in a neighborhood of $0\in\ch_\vr^+$. And in particular, we can see that $I_{\vr,u}(0)=I_{\vr,u}'(0)=0$ and $I_{\vr,u}''(0)>0$. Then, to complete the proof, it is sufficient to show that
\begin{\equ}\label{I-equivalent}
I_\vr'(u)[u]=0, \ u\neq0 \quad \Longrightarrow \quad I_\vr''(u)[u,u]<0.
\end{\equ}
For simplicity, let us denote $\Psi_\vr:\ch\to\R$ by $\Psi_\vr(\psi)
=\frac1{\vr^m}\int_M F(|\psi|)d\vol_\ig$ and set
$\psi=u+\chi_\vr(u)$ and $w=\chi_\vr'(u)[u]-\chi_\vr(u)$. By using
$\cl_\vr'(u+\chi_\vr(u))|_{\ch_\vr^-}\equiv0$, we see that \eqref{I-equivalent} is a direct consequence of the following computation:
\begin{eqnarray}\label{eee1}
I_\vr''(u)[u,u]&=&\cl_\vr''(\psi)[u+\chi_\vr'(u)[u],u]=\cl_\vr''(\psi)[\psi+w,\psi+w]  \nonumber\\
&=&\cl_\vr''(\psi)[\psi,\psi]+2\cl_\vr''(\psi)[\psi,w]
+\cl_\vr''(\psi)[w,w]  \nonumber\\
&=&I_\vr'(u)[u]+\big(\Psi_\vr'(\psi)[\psi]-\Psi_\vr''(\psi)[\psi,\psi]\big)
+2\big(\Psi_\vr'(\psi)[w]-\Psi_\vr''(\psi)[\psi,w] \big) \nonumber\\
& &\quad -\Psi_\vr''(\psi)[w,w]
-\|w\|_\vr^2 \nonumber\\
& \leq& I_\vr'(u)[u]-\frac1{\vr^m}\int_M \frac{f(|\psi|)f'(|\psi|)|\psi|^3}{f(|\psi|)+f'(|\psi|)|\psi|}
 d\vol_\ig-\|w\|_\vr^2.
\end{eqnarray}
Since $f(s)f'(s)s\leq(f(s)+f'(s)s)^2$ for all $s\geq0$, and the equality holds iff $s=0$, we have
\[
0<\frac{f(|\psi|)f'(|\psi|)|\psi|^3}{f(|\psi|)+f'(|\psi|)|\psi|}\leq
f(|\psi|)|\psi|^2+f'(|\psi|)|\psi|^3 \quad \text{for } \psi\neq0.
\]
Therefore, the integration in \eqref{eee1} is well-defined due to $(f_1)$-$(f_3)$.
\end{proof}

The above lemma indicates that, for each $u\in\ch_\eps^+\setminus\{0\}$, the function $I_{\vr,u}(\cdot)$ has at most one critical point $t_{\eps,u}\in(0,+\infty)$. Due to \eqref{eeee1}, one checks that such $t_{\eps,u}$ exists for every $u$. And then we can define a natural constraint for the reduced functional $I_\vr$ as
\[
\msn_\vr:=\big\{ u\in\ch_\vr^+\setminus\{0\}:\, I_\vr'(u)[u]=0 \big\}.
\]
Lemma \ref{reduction2} also implies $\msn_\vr$ is a smooth submanifold of codimension $1$ in $\ch_\vr^+$. And consequently, the critical point found in Proposition \ref{reduction1} $(3)$
can be characterized by
\begin{\equ}\label{ga-vr}
\mu_\vr:=\cl_\vr(\bar\psi_\vr)=\inf_{u\in\ch_\vr^+\setminus\{0\}}
\max_{\psi\in \R u\op \ch_\vr^-}\cl_\vr(\psi)
=\inf_{u\in\ch_\vr^+\setminus\{0\}}\max_{t>0} I_\vr(tu)
=\inf_{u\in\msn_\vr}I_\vr(u).
\end{\equ}
For later purposes,  it is worth to remind that \eqref{eeee2} implies the existence of some $\tau_0>0$ independent of $\vr$ such that $\mu_\vr\geq\tau_0$.

\medskip

In what follows, we intend to pass to the limit $\vr\to0$ and consider the asymptotic behavior of the min-max level $\mu_\vr$. The idea is to provide an upper bound estimate around an arbitrary $(P.S.)$-sequence, so that one may substitute certain test spinors
in the functional $\cl_\vr$.

Without loss of generality, we assume that $\{\phi_\vr\}\subset\ch$ is an arbitrary sequence such that
\begin{\equ}\label{assumption0}
c_1\leq \cl_\vr(\phi_\vr)\leq c_2 \quad \text{and} \quad
\|\cl_\vr'(\phi_\vr)\|_\vr\to0
\end{\equ}
as $\vr\to0$ for some constants $c_1,c_2>0$. Here, we will
identify the dual space $\ch^*$ with $\ch$.

\begin{Lem}\label{a1}
Under \eqref{assumption0}, we have:
\begin{itemize}
\item[$(1)$] $\|\phi_\vr\|_\vr$ is uniformly bounded in $\vr$.

\item[$(2)$] $\|\phi_\vr^- - \chi_\vr(\phi_\vr^+)\|_\vr\leq O\big( \|\cl_\vr'(\phi_\vr)\|_\vr \big)$
    as $ \vr\to0$.

\item[$(3)$] $I_\vr'(\phi_\vr^+)\to0$ as $\vr\to0$ in the dual space of $\ch_\vr^+$.
\end{itemize}
\end{Lem}
\begin{proof}
For the boundedness, we recall that Lemma \ref{embeding lemma} implies that
the embedding constant for $\ch\hookrightarrow L^{p^*}$ is independent of
$\vr$, and hence the arguments in Lemma \ref{PS-condition} can be employed.

For $(2)$, let us first set $z_\vr=\phi_\vr^++\chi_\vr(\phi_\vr^+)$ and $v_\vr=\phi_\vr^--\chi_\vr(\phi_\vr^+)$. Then we have $v_\vr\in \ch_\vr^-$
and, by the definition of $\chi_\vr$,
\[
0=\cl_\vr'(z_\vr)[v_\vr]=-\inp{\chi_\vr(\phi_\vr^+)}{v_\vr}_\vr-\frac1{\vr^m}\Real\int_M
f( |z_\vr|)(z_\vr,v_\vr)d\vol_\ig.
\]
Since $\|\cl_\vr'(\phi_\vr)\|_\vr\to 0$ as $\vr\to0$, it follows that
\[
o(\|v_\vr\|_\vr)=\cl_\vr'(\phi_\vr)[v_\vr]=-\inp{\phi_\vr^-}{v_\vr}-\frac1{\vr^m}
\Real\int_M f (|\phi_\vr|)(\phi_\vr,v_\vr)d\vol_\ig.
\]
And hence, we get
\begin{\equ}\label{e1}
o(\|v_\vr\|_\vr)=\|v_\vr\|_\vr^2+\frac1{\vr^m}
\Real\int_M f( |\phi_\vr|)(\phi_\vr,v_\vr)d\vol_\ig
 -\frac1{\vr^m}\Real\int_M
f(|z_\vr|)(z_\vr,v_\vr)d\vol_\ig.
\end{\equ}
Since the map $\psi\to F(|\psi|)$ is convex by $(f_1)$, we have
\[
\frac1{\vr^m}
\Real\int_M f(|\phi_\vr|)(\phi_\vr,v_\vr)d\vol_\ig
-\frac1{\vr^m}\Real\int_M
f(|z_\vr|)(z_\vr,v_\vr)d\vol_\ig\geq 0.
\]
Thus, from \eqref{e1}, we can infer that $\|v_\vr\|_\vr\leq O\big( \|\cl_\vr'(\phi_\vr)\|_\vr \big)$
as $\vr\to0$.

In order to check $(3)$ we compute $I_\vr'(\phi_\vr^+)=\cl_\vr'\big(\phi_\vr^++\chi_\vr(\phi_\vr^+)\big)$,
which implies that $\|I_\vr'(\phi_\vr^+)\|_\vr\to0$ as $\vr\to0$ is a
direct consequence of $(2)$ and the $C^2$ smoothness of $\cl_\vr$.
\end{proof}

Next, let us introduce the functional $K_\vr: \ch_\vr^+\to\R$ by $K_\vr(u)=I_\vr'(u)[u]$. Then, it is clear that $K_\vr$ is $C^1$ and its derivative is given by the formula
\[
K_\vr'(u)[w]=I_\vr'(u)[w]+I_\vr''(u)[u,w]
\]
for $u,w\in\ch_\vr^+$. We also have $\msn_\vr=K_\vr^{-1}(0)\setminus\{0\}$.
Moreover, by \eqref{eee1}, there holds
\[
K_\vr'(u)[u]\leq 2 K_\vr(u)-\frac1{\vr^m}\int_M \frac{f(|\psi|)f'(|\psi|)|\psi|^3}{f(|\psi|)+f'(|\psi|)|\psi|}
 d\vol_\ig, \quad \text{for } u\in \ch_\vr^+
\]
where $\psi=u+\chi_\vr(u)$. By virtue of $(f_3)$, one checks easily that $\frac{f(s)f'(s)s}{f(s)+f'(s)s}\geq \frac\theta{\theta+1}f(s)$ for $s>0$. Thus the above estimate implies
\begin{\equ}\label{eee2}
K_\vr'(u)[u]\leq 2 K_\vr(u)-\frac\theta{(\theta+1)\vr^m}\int_M
f(|u+\chi_\vr(u)|)|u+\chi_\vr(u)|^2
 d\vol_\ig,
\end{\equ}
for $ u\in \ch_\vr^+$.
\begin{Prop}\label{estimate prop}
For the sequence $\{\phi_\vr\}$ in \eqref{assumption0}, there exists
$\{t_\vr\}\subset\R$ such that $t_\vr\phi_\vr^+\in\msn_\vr$ and
$|t_\vr-1|\leq O\big(\|I_\vr'(\phi_\vr^+)\|_\vr\big)$.
\end{Prop}
\begin{proof}
We begin with the observation that $F(s)\leq\frac1{\theta+2}f(s)s^2$ for all $s\geq0$. Due to the condition \eqref{assumption0} and Lemma \ref{a1} $(3)$, it follows directly that
\begin{\equ}\label{o1}
\liminf_{\vr\to0}\frac1{\vr^m}\int_M f \big(| \phi_\vr^++\chi_\vr(\phi_\vr^+)|\big)| \phi_\vr^++\chi_\vr(\phi_\vr^+)|^2d\vol_\ig
\geq c_0
\end{\equ}
for some constant $c_0>0$. Let us set $\eta_\vr:(0,\infty)\to\R$ by
$\eta_\vr(t)=K_\vr(t\phi_\vr^+)$. One easily checks that
$t\eta_\vr'(t)=K_\vr'(t\phi_\vr^+)[t\phi_\vr^+]$ for all $t>0$. Hence, by \eqref{eee2} and Taylor's expansion, we get
\begin{\equ}\label{e2}
t\eta_\vr'(t)\leq 2 \eta_\vr(1)-\frac\theta{(\theta+1)\vr^m}\int_M
f \big(| \phi_\vr^++\chi_\vr(\phi_\vr^+)|\big)| \phi_\vr^++\chi_\vr(\phi_\vr^+)|^2d\vol_\ig + C|t-1|
\end{\equ}
for $t$ close to $1$ with some $C>0$ independent of $\vr$. Here we have used the uniform boundedness of $\eta_\vr'(t)$ on bounded intervals.

Noticing that $\eta_\vr(1)=I_\vr'(\phi_\vr^+)[\phi_\vr^+]\to0$ as $\vr\to0$,
we conclude from \eqref{o1} and \eqref{e2} that there exists a small constant $\de>0$ such that
\[
\eta_\vr'(t)\leq-\de \text{ for all } t\in(1-\de,1+\de) \text{ and } \vr \text{ small enough}.
\]
Moreover, notice that $K_\vr(u)$ equals to the value of $I_{\vr,u}'(1)$, it follows from Lemma \ref{reduction2} that $\eta_\vr(1-\de)>0$ and
$\eta_\vr(1+\de)<0$. Then, by the Inverse Function Theorem, $t_\vr:=\eta_\vr^{-1}(0)$
exists and
\[
u_\vr:=t_\vr\phi_\vr^+\in\msn_\vr\cap \R^+\phi_\vr^+
\]
is well-defined for all $\vr$ small enough. Furthermore, since $|\eta_\vr'(t)^{-1}|$
is bounded by a constant, say $c_\de>0$, on $(1-\de,1+\de)$, we consequently get
\[
\|u_\vr-\phi_\vr^+\|_\vr=
|\eta_\vr^{-1}(0)-\eta_\vr^{-1}(K_\vr(\phi_\vr^+))|\cdot\|\phi_\vr^+\|_\vr
\leq c_\de |K_\vr(\phi_\vr^+)|\cdot\|\phi_\vr^+\|_\vr.
\]
Now the conclusion follows from $K_\vr(\phi_\vr^+)\leq O\big(\|I_\vr'(\phi_\vr^+)\|_\vr\big)$.
\end{proof}

\begin{Cor}\label{key corollary}
For the sequence $\{\phi_\vr\}$ in \eqref{assumption0}, there exists
$\{u_\vr\}$ such that $u_\vr\in\msn_\vr$ and
$\|\phi_\vr-u_\vr-\chi_\vr(u_\vr)\|_\vr\leq O(\|\cl_\vr'(\phi_\vr)\|_\vr)$.
Moreover,
\[
\max_{t>0}I_\vr(t\phi_\vr^+)=I_\vr(u_\vr)\leq \cl_\vr(\phi_\vr)
+O\big( \|\cl_\vr'(\phi_\vr)\|_\vr^2 \big).
\]
\end{Cor}
\begin{proof}
To see this, let $u_\vr=t_\vr\phi_\vr^+$ be as in Proposition \ref{estimate prop}
and set $z_\vr=\phi_\vr^++\chi_\vr(\phi_\vr^+)$. Then one obtains from Lemma \ref{a1} that
\begin{\equ}\label{e3}
\aligned
\|\phi_\vr-u_\vr-\chi_\vr(u_\vr)\|_\vr&\leq \|\phi_\vr-z_\vr\|_\vr +|t_\vr-1|\cdot\|\phi_\vr^+\|_\vr + \|\chi_\vr(\phi_\vr^+)-\chi_\vr(u_\vr)\|_\vr \\
&\leq O\big( \|\cl_\vr'(\phi_\vr)\|_\vr \big)+O\big( \|I_\vr'(\phi_\vr^+)\|_\vr \big)
\endaligned
\end{\equ}
where we have used an easily checked inequality
\[
\|\chi_\vr(\phi_\vr^+)-\chi_\vr(u_\vr)\|_\vr\leq \|\chi_\vr'(\tau \phi_\vr^+)\|_{\ch_\vr^+\to\ch_\vr^-}\cdot
\|\phi_\vr^+-u_\vr\|_\vr= O(|t_\vr-1|)
\]
for some $\tau$ between $t_\vr$ and $1$.
Observing that $I_\vr'(\phi_\vr^+)=\cl_\vr'(z_\vr)$, and using the $C^2$ smoothness
of $\cl_\vr$, we have
\[
\|I_\vr'(\phi_\vr^+)\|_\vr=\|\cl_\vr'(z_\vr)\|_\vr\leq \|\cl_\vr'(\phi_\vr)\|_\vr
+O(\|\phi_\vr-z_\vr\|_\vr)=O(\|\cl_\vr'(\phi_\vr)\|_\vr).
\]
This together with \eqref{e3} implies
\[
\|\phi_\vr-u_\vr-\chi_\vr(u_\vr)\|_\vr\leq O(\|\cl_\vr'(\phi_\vr)\|_\vr).
\]

Now, by Talyor's expansion, we can obtain
\[
\aligned
\cl_\vr(\phi_\vr)&=\cl_\vr(u_\vr+\chi_\vr(u_\vr))
+\cl_\vr'(u_\vr+\chi_\vr(u_\vr))[\phi_\vr-u_\vr-\chi_\vr(u_\vr)]
+O\big(\|\cl_\vr'(\phi_\vr)\|_\vr^2\big)  \\
&= I_\vr(u_\vr)+I_\vr'(u_\vr)[\phi_\vr^+-u_\vr]+O\big(\|\cl_\vr'(\phi_\vr)\|_\vr^2\big) .
\endaligned
\]
Since $u_\vr=t_\vr\phi_\vr^+\in\msn_\vr$, we have $I_\vr'(u_\vr)[\phi_\vr^+-u_\vr]\equiv0$
and this implies the last estimate.
\end{proof}

\section{Proof of the main results}\label{Proof of the main results sec}
\subsection{The equation on Euclidean spaces: the bubbles}\label{limit equ sec}

We consider solutions to the equation
\begin{\equ}\label{limit equ}
\tilde D_{\ig_{\R^m}}\psi + a \om_\C\cdot_{\ig_{\R^m}}\psi = f(|\psi|)\psi
\quad \text{on } \R^m
\end{\equ}
belonging to the class $W^{\frac12,2}(\R^m,\tilde\mbs(\R^m))$, where
\[
\tilde\mbs(\R^m)=\left\{
\aligned
&\mbs(\R^m)\op\mbs(\R^m) &\quad &  m \text{ is odd}, \\
&\qquad \mbs(\R^m) &\quad & m \text{ is even},
\endaligned \right.
\]
$\tilde D_{\ig_{\R^m}}=D_{\ig_{\R^m}}\op -D_{\ig_{\R^m}}$
if $m$ is odd and $\tilde D_{\ig_{\R^m}}=D_{\ig_{\R^m}}$ if $m$ is even.
These solutions will correspond to "bubbles" or test spinors for our variational problem. We also assume the nonlinear function $f$ satisfies conditions $(f_1)$-$(f_3)$.

First of all, let us set $A=\tilde D_{\ig_{\R^m}}+a\om_\C\cdot_{\ig_{\R^m}}$. By a straightforward calculation we see that $A$ is a self-adjoint operator on $L^2$ with spectrum $Spec(A)=(-\infty,-a]\cup[a,+\infty)$. Following Amann \cite{Amann} let $(E_\lm)_{\lm\in\R}$ be the spectral resolution of $A$ and define the orthogonal projections by
\[
P=\int_{-\infty}^0 d E_\lm, \quad  Q=\int_0^{\infty} d E_\lm.
\]
Then the decomposition of $\ce=W^{\frac12,2}(\R^m,\tilde\mbs(\R^m))=\ce^+\op\ce^-$ is induced by
\[
\ce^-=\ce\cap P (L^2) \quad \text{and} \quad \ce^+=\ce\cap Q (L^2).
\]
We introduce the following operators
\[
S=\int_{-\infty}^0 |\lm|^{\frac12}d E_\lm \quad \text{and} \quad
T=\int_{0}^\infty |\lm|^{\frac12} d E_\lm.
\]
and the new inner product on $\ce$
\[
\inp{\psi}{\va}=\Real\big( (S+T)\psi, (S+T)\va \big)_2, \quad
\psi,\va\in\ce
\]
with $\|\cdot\|$ denoting the corresponding norm. We easily see that
\eqref{limit equ} is the Euler-Lagrange equation of the functional
\begin{\equ}\label{limit equ functional}
\Phi(\psi)=\frac12\big(\|Q\psi\|^2-\|P\psi\|^2\big)
-\int_{\R^m}F(|\psi|)dx.
\end{\equ}

\begin{Lem}\label{H5}
If $\{\psi_n\}\subset \ce$ is a bounded sequence such that
\[
\Phi'(\psi_n)\to0 \quad \text{and} \quad
\liminf_{n\to\infty}\int_{\R^m}f(|\psi_n|)|\psi_n|^2dx>0.
\]
then there exists $\psi\neq0$ with $\Phi'(\psi)=0$.
\end{Lem}
\begin{proof}
Let $B_R^0$ denote the open ball of radius $R$ centered at the origin. If
\[
\lim_{n\to\infty}\sup_{y\in \R^m}\int_{y+B_R^0} |\psi_n|^2dx=0, \quad \forall R>0,
\]
then a result of Lions \cite{Lions} implies $\psi_n\to0$ in $L^{q}$ for all $q\in(2,m^*)$
and therefore $\int_{\R^m}f(|\psi_n|)|\psi_n|^2dx\to0$, which is a contradiction.

Passing to a subsequence, we have
\[
\liminf_{n\to\infty}\int_{y_n+B_R^0}|\psi_n|^2dx>0
\]
for some $R>0$ and $\{y_n\}\subset\R^m$. Using the invariance of the operator $A$
under translations, we can find $R>0$ and a new sequence $\{\tilde\psi_n\}$ such that
\[
\Phi'(\tilde\psi_n)\to0 \quad \text{and} \quad
\liminf_{n\to\infty}\int_{B_R^0}|\tilde\psi_n|^2 dx>0.
\]
Up to a subsequence if necessary, we have $\tilde\psi_n\rightharpoonup \psi$
and the compact embedding $\ce\hookrightarrow L^2_{loc}$ shows that $\psi\neq0$. By taking the limit in $\Phi'(\tilde\psi_n)\to0$, we obtain $\Phi'(\psi)=0$ as desired.
\end{proof}

\begin{Cor}
There exists a nontrivial solution $\psi\in\ce$ to Eq. \eqref{limit equ}.
\end{Cor}
\begin{proof}
By Lemma \ref{H5} and the boundedness argument in Lemma \ref{PS-condition}, this is a direct consequence of \cite[Theorem 2.1]{BJS}.
\end{proof}

Now we may define
\begin{\equ}\label{C0}
\mu_0=\inf\big\{ \Phi(\psi): \, \psi\in\ce\setminus\{0\} \text{ s.t. }
\Phi'(\psi)=0 \big\}.
\end{\equ}
Since the super-quadratic part in \eqref{limit equ functional} has subcritical growth at infinity, one easily sees that $\mu_0>0$ is attained. In particular, analogously to Proposition \ref{reduction1} and Lemma \ref{reduction2}-Corollary \ref{key corollary}, the following reduction principle holds.

\begin{Lem}\label{reduction limit equ}
\begin{itemize}
\item[$(1)$] There exists a $C^1$ map $h:\ce^+\to\ce^-$ such that
$\displaystyle\Phi(u+h(u))=\max_{v\in\ce^-}\Phi(u+v)$.

\item[$(2)$] The critical points of the functional $J(u)=\Phi(u+h(u))$ and those of $\Phi$ are in one-to-one correspondence via the map $u\mapsto u+h(u)$.

\item[$(3)$] For each $u\in\ce^+\setminus\{0\}$, the map $t\mapsto J(tu)$ has only one maximum on $(0,+\infty)$ and
$\displaystyle \mu_0=\inf_{u\in\ce^+\setminus\{0\}}
\max_{t>0}J(tu)$.

\item[$(4)$] For any bounded sequence $\{z_n\}\in\ce$ such that
$\Phi(z_n)\to c>0$ and $\Phi'(z_n)\to0$, there holds
\[
\max_{t>0}J(t z_n^+)\leq \Phi(z_n)+O(\|\Phi'(z_n)\|^2).
\]
\end{itemize}
\end{Lem}

From elliptic estimates and and bootstrap arguments, we deduce that (weak) solutions of \eqref{limit equ} with bounded energy are uniformly bounded in $\cap_{q\geq2}W^{1,q}(\R^m,\tilde\mbs(\R^m))$. Moreover, we have the following

\begin{Lem}\label{exponential}
Setting $\cb=\big\{\psi\in\ce:\, \Phi(\psi)=\mu_0,\, \Phi'(\psi)=0, |\psi(0)|=\max_{\R^m}|\psi|\big\}$, the following holds.
\begin{itemize}
\item[$(1)$] $\cb$ is compact in $W^{1,2}(\R^m,\tilde\mbs(\R^m))$.

\item[$(2)$] There exist $C,c>0$ such that
$|\psi(x)|\leq C \exp(-c|x|)$ for all $\psi\in\cb$.
\end{itemize}
\end{Lem}
\begin{proof}
Clearly, $\cb$ is closed in $\ce$. We show that an arbitrary sequence $\psi_n\in\cb$, $n\in\N$, in $B$ has a convergent subsequence.

In fact, since $\{\psi_n\}$ is bounded, we have $\psi_n\rightharpoonup \psi_0$ along a subsequence in $\ce$ with clearly $\psi_0\in\cb$. Hence, one has $\psi_n\to\psi_0$ in $L_{loc}^q$ for $q\in[2,m^*)$. Moreover, by the fact that
\[
\mu_0=\int_{\R^m}\frac12f(|\psi_n|)|\psi_n|^2-F(|\psi_n|)dx
=\int_{\R^m}\frac12f(|\psi_0|)|\psi_0|^2-F(|\psi_0|)dx,
\]
it is easy to see that for every $\epsilon>0$ there is $R>0$ such that
\begin{\equ}\label{X0}
\limsup_{n\to\infty}\int_{|x|\geq R}\frac12f(|\psi_n|)|\psi_n|^2-F(|\psi_n|)dx\leq\epsilon.
\end{\equ}
Setting $z_n=\psi_n-\psi_0$ we obtain, using $\Phi'(\psi_n)=\Phi'(\psi_0)=0$,
\[
\inp{Q\psi_n}{Qz_n}+\inp{P\psi_n}{Pz_n}-\Real\int_{\R^m}f(|\psi_n|)(\psi_n,Qz_n-Pz_n)dx=0
\]
and
\[
\inp{Q\psi_0}{Qz_n}+\inp{P\psi_0}{Pz_n}-\Real\int_{\R^m}f(|\psi_0|)(\psi_0,Qz_n-Pz_n)dx=0.
\]
Hence, we get
\begin{\equ}\label{X1}
\|z_n\|^2=\Real\int_{\R^m}f(|\psi_n|)(\psi_n,Qz_n-Pz_n)dx+o_n(1),
\end{\equ}
where we have used $\Real\int_{\R^m}f(|\psi_0|)(\psi_0,Qz_n-Pz_n)dx\to0$ as $n\to\infty$ which holds because $z_n\rightharpoonup 0$ in $L^q$ for $q\in[2,m^*]$. Recall that, by $(f_1)$-$(f_3)$, for arbitrarily small $\de>0$ there exists $c_\de>0$ such that $f(s)s\leq \de s+c_\de(f(s)s^2)^{\frac{p-1}{p}}$ and $F(s)\leq\frac1{\theta+2}f(s)s^2$ for all $s\geq0$. Thus, \eqref{X0} and \eqref{X1} imply
\[
\aligned
\|z_n\|^2&\leq\de|\psi_n|_2|Qz_n-Pz_n|_2+c_\de\Big(\int_{|x|\geq R} f(|\psi_n|)|\psi_n|^2dx\Big)^{\frac{p-1}{p}}|Qz_n-Pz_n|_p+o_n(1) \\
&\leq\de C\|z_n\|+ C_\de \big(\epsilon+o_n(1)\big)^{\frac{p-1}{p}}\|z_n\|+o_n(1).
\endaligned
\]
Due to the arbitrariness of $\de,\epsilon>0$, one sees $\|z_n\|\to0$ as $n\to\infty$.

Now we prove the compactness in $W^{1,2}(\R^m,\tilde\mbs(\R^m))$. As a consequence of the equation (recall $A=\tilde D_{\ig_{\R^m}}+a\om_\C\cdot_{\ig_{\R^m}}$)
\[
A\psi_n=f(|\psi_n|)\psi_n \quad \text{and} \quad
A\psi_0=f(|\psi_0|)\psi_0,
\]
we have
\[
\aligned
|A(\psi_n-\psi_n)|_2&=\big|f(|\psi_n|)\psi_n-f(|\psi_0|)\psi_0\big|_2\\
&\leq\big|f(|\psi_n|)(\psi_n-\psi_0)\big|_2+\big|(f(|\psi_n|)-f(|\psi_0|))\psi_0 \big|_2.
\endaligned
\]
From $|\psi_n|_\infty\leq C$ and $\psi_n\to\psi_0$ in $\ce$ we deduce
\[
\big|f(|\psi_n|)(\psi_n-\psi_0)\big|_2\leq f(C)|\psi_n-\psi_0|_2=o_n(1), \quad \text{as }n\to\infty
\]
and
\[
\aligned
\int_{\R^m}\big|(f(|\psi_n|)-f(|\psi_0|))\psi_0 \big|^2dx
&=\int_{|x|\leq R} \big|(f(|\psi_n|)-f(|\psi_0|))\psi_0 \big|^2dx+o_R(1) \\
&=o_n(1)+o_R(1)
\endaligned
\]
as $n\to\infty$ because $|\psi_0(x)|\to0$ as $|x|\to\infty$. Therefore, we get $|A(\psi_n-\psi_n)|_2\to0$, i.e. $\psi_n\to\psi_0$ in $W^{1,2}(\R^m,\tilde\mbs(\R^m))$.

\medskip

To see the exponential decay, we rewrite \eqref{limit equ} as
\[
\tilde D_{\ig_{\R^m}}\psi =- a \om_\C\cdot_{\ig_{\R^m}}\psi + f(|\psi|)\psi.
\]
Applying the operator $\tilde D_{\ig_{\R^m}}$ to both sides of the above equation and noting that $\tilde D_{\ig_{\R^m}}^2=-\De$, we find
 \[
 \De \psi=\big( a^2-f(|\psi|)^2 \big)\psi-\nabla f(|\psi|)\cdot_{\ig_{\R^m}}\psi.
 \]
Now using the fact that
\[
\De|\psi|^2=2\Real(\De\psi,\psi)+2|\nabla\psi|^2
\]
and that $\Real(\nabla f(|\psi|)\cdot_{\ig_{\R^m}}\psi,\psi)\equiv0$, we obtain
\[
\De|\psi|^2=2\big( a^2-f(|\psi|)^2 \big)|\psi|^2+2|\nabla\psi|^2\geq2\big( a^2-f(|\psi|)^2 \big)|\psi|^2.
\]
Now for $\psi\in\cb$, by $|\psi(x)|\to0$ as $|x|\to\infty$, we may take $R>0$ large enough so that
\[
\De|\psi|^2\geq a^2|\psi|^2
\]
for all $|x|\geq R$. 

Let $\Ga(x)=e^{-a|x|}$. One checks easily that
\[
\De \Ga-a^2\Ga<0 ,\quad \text{for } |x|>0.
\]
By taking $C>0$ be such that $|\psi(x)|^2\leq  C\cdot\Ga(x)$ holds on $|x|=R$, we may consider $U=|\psi|^2-C\cdot \Ga$ and get
\[
\De U=\De |\psi|^2-C\cdot\De\Ga> a^2 U.
\]
By elliptic estimates and the comparison principle, we can easily conclude that $U(x)\leq 0$ for all $|x|\geq R$. Hence we obtain the exponential decay of $|\psi(x)|$ at infinity. Finally, thanks to the compactness of $\cb$, we see that the exponential decay holds uniformly for all $\psi\in\cb$.
\end{proof}

%

\subsection{Bourguignon-Gauduchon trivialization}

Our proof relies on the construction of a test spinor on $M$ in order to show the concentration behavior under the conditions $(f_1)$-$(f_3)$. The test spinor comes from a spinor on $\R^m$ being cut-off and transplanted to $M$ so that it has support in a small neighborhood of an arbitrary point $y\in M$. We first need to recall a construction from the paper \cite{AGHM} of Ammann et al.

To begin with, we fix a spinor field $\psi\in\cb$ arbitrarily. Let $r<\textit{inj}(M)/2$ where $\textit{inj}(M)>0$ is the injectivity radius of $M$, and let $\eta\in C_c^\infty(\R^m,[0,1])$ be such that $|\nabla\eta|\leq2/r$, $\eta(x)=1$ for $|x|\leq r$ and $\eta(x)=0$ for $|x|\geq2r$. Then, we define $\va_\vr:\R^m\to\mbs_m$ by
\begin{\equ}\label{spinor Rm}
\va_\vr(x)=\eta(x)\psi_\vr(x) \quad \text{where} \quad
\psi_\vr(x)=\psi(x/\vr).
\end{\equ}

In order to transplant the test spinor on $M$, we recall the Bourguignon-Gauduchon-trivialization. Here we fix $y\in M$ arbitrarily, and let $(x_1,\dots,x_m)$ be the normal coordinates given by the exponential map
\[
  \exp_y: \R^m\cong T_{y}M\supset U \to V\subset M,\quad x \mapsto \xi = \exp_y(x).
\]
For $\xi\in V$, let $G(\xi)=(\ig_{ij}(\xi))_{ij}$ denote the corresponding metric at $\xi$. Since $G(\xi)$ is symmetric and positive definite, the square root
$B(\xi)=(b_{ij}(\xi))_{ij}$ of $G(\xi)^{-1}$ is well defined, symmetric and positive definite. It can be thought of as linear isometry
\[
  B(\xi): (\R^m\cong T_{\exp_y^{-1}(\xi)}U,\ig_{\R^m}) \to (T_pV,\ig).
\]
We obtain an isomorphism of $SO(m)$-principal bundles:
\begin{displaymath}
\xymatrix{
  P_{SO}(U,\ig_{\R^m}) \ar[r]^{\displaystyle\phi}  \ar[d] & P_{SO}(V,\ig) \ar[d] \\
T_yM \supset U\ar[r]^{\ \ \displaystyle\exp_y} & V \subset M}
\end{displaymath}
where $\phi\{v_1,\dots,v_m\} = \{Bv_1,\dots,Bv_m\}$ for an oriented frame $\{v_1,\dots,v_m\}$ on $U$. Notice that $\phi$ commutes with the right action of $SO(m)$, hence it induces an isomorphism of spin structures:
\begin{displaymath}
\xymatrix{
Spin(m)\times U  =P_{Spin}(U,\ig_{\R^m}) \ar[r] \ar[d] & P_{Spin}(V,\ig) \subset P_{Spin}(M) \ar[d]\\
T_yM\supset U\ar[r]^{\ \ \displaystyle\exp_y} & V \subset M}
\end{displaymath}
Thus we obtain an isomorphism between the spinor bundles $\mbs(U)$ and $\mbs(V)$:
\begin{equation}\label{spin-iso}
  \mbs(U) := P_{Spin}(U,\ig_{\R^m})\times_\rho \mbs_m \longrightarrow \mbs(V) := P_{Spin}(V,\ig)\times_\rho \mbs_m \subset \mbs(M)
\end{equation}
where $(\rho,\mbs_m)$ is the complex spin representation.

Let $\{\pa_1,\dots,\pa_m\}$ be the canonical frame on the Euclidean space, where $\pa_i=\frac{\pa}{\pa x_i}$.
Setting $e_i=B(\pa_i)=\sum_{j}b_{ij}\pa_j$ we obtain an orthonormal frame $\{e_1,\dots, e_m\}$ of $(TV,\ig)$. Via \eqref{spin-iso}, we also have
\[
e_i\cdot_\ig\bar\psi=B(\pa_i)\cdot_\ig\bar\psi=\ov{\pa_i\cdot_{\ig_{\R^m}}\psi}
\quad \text{for } \psi\in\mbs_m.
\]

Now a spinor field $\va\in\Ga(\tilde\mbs(U))$ corresponds via the isomorphim \eqref{spin-iso} to a spinor $\bar\va\in\Ga(\tilde\mbs(V))$, and we will keep this notation for various spinor fields to represent such correspondence. In particular, since the spinors $\va_\vr\in\Ga(\tilde\mbs(U))$ have compact support in $U$ they correspond to spinors $\bar\va_\vr\in\Ga(\tilde\mbs(M))$ with compact support in $V$.

In the sequel, in order to simplify the notation, we use $\nabla$ and $\bar\nabla$, respectively, for the Levi-Civita connections on $(TU,\ig_{\R^m})$ and $(TV,\ig)$ and for the natural lifts of these connections to the spinor bundles $\tilde\mbs(U)$ and $\tilde\mbs(V)$, respectively. By abuse of notation, we write $D$ and $\bar D$ instead of the Dirac operators acting on $\Ga(\tilde\mbs(U))$ and $\Ga(\tilde\mbs(V))$, respectively. By \cite[Proposition 3.2]{AGHM} there holds
\begin{\equ}\label{cut-off spinor identity}
  \bar D \bar\va_\vr = \ov{D\va_\vr}+W\cdot_\ig\bar\va_\vr + X\cdot_\ig\bar\va_\vr + \sum_{i,j}(b_{ij}-\de_{ij})\ov{\pa_i\cdot_{\ig_{\R^m}}\nabla_{\pa_j}\va_\vr}
\end{\equ}
with $W\in\Ga(Cl(TV))$ and $X\in\Ga(TV)$ given by
\[
  W = \frac14\sum_{\substack{i,j,k \\ i\neq j\neq k\neq i}}\sum_{\al,\bt} b_{i\al}(\pa_{\al}b_{j\bt})b_{\bt k}^{-1}e_i\cdot_\ig e_j\cdot_\ig e_k,
\]
and
\[
  X = \frac14\sum_{i,k}\big( \bar\Ga_{ik}^i-\bar\Ga_{ii}^k\big)e_k = \frac12\sum_{i,k} \bar\Ga_{ik}^i e_k;
\]
here $(b_{ij}^{-1})_{ij}$ denotes the inverse matrix of $B$, and $\bar\Ga_{ij}^k:=\ig(\bar\nabla_{e_i}e_j,e_k)$. In what follows we identify $x=(x_1,\dots,x_m)\in U\subset\R^m$ with $\sum_{i}x_ie_i\in T_yM$ for notational convenience. As remarked in \cite{LP, Chavel}, in the neighborhood of $y$, the metric $\ig$ and its determinant have the following expansion:
\begin{\equ}\label{metric development}
\ig_{ij}(\exp_yx)=\de_{ij}-\frac13\rr_y(e_i,x,x,e_j)+O(|x|^3), 
\end{\equ}
\begin{\equ}\label{det-g development}
\sqrt{\text{det}\,\ig}(\exp_yx)=1-\frac16\text{Ric}_y(x,x)+O(|x|^3)
\end{\equ}
where
\[
\rr(e_i,e_j,e_k,e_l)=\ig(\nabla_{e_i}\nabla_{e_j}e_k,e_l)-\ig(\nabla_{e_j}\nabla_{e_i}e_k,e_l)-\ig(\nabla_{[e_i,e_j]}e_k,e_l)
\]
is the Riemannian curvature tensor and $\text{Ric}(v,w)=\sum_{i=1}^m\rr(e_i,v,w,e_i)$ is the Ricci curvature. Observing that $B=(G^{-1})^{\frac12}$, as in \cite{AGHM}, we have
\begin{\equ}\label{WX}
b_{ij}=\de_{ij}+\frac16\rr_y(e_i,x,x,e_j)+O(|x|^3) , \quad W=O(|x|^3)\quad \text{and} \quad X=O(|x|).
\end{\equ}

The main results of this subsection will be the following two lemmas.

\begin{Lem}\label{estimate for L-derivative}
Let $\bar\va_\vr$ be as in \eqref{spinor Rm}, then $\|\cl_\vr'(\bar\va_\vr)\|_\vr\leq O(\vr^2)$ as $\vr\to0$.
\end{Lem}
\begin{proof}
By the definition of $\va_\vr$ in \eqref{spinor Rm}, together with \eqref{cut-off spinor identity}, one checks easily that
\begin{\equ}\label{X2}
\vr D\va_\vr=\vr\nabla\eta\cdot_{\ig_{\R^m}}\psi_\vr+\vr\eta D\psi_\vr
=\vr\nabla\eta\cdot_{\ig_{\R^m}}\psi_\vr-a\eta\om_\C\cdot_{\R^m}\psi_\vr+\eta f(|\psi_\vr|)\psi_\vr
\end{\equ}
and
\[
\vr\bar D\bar\va_\vr+a\om_\C\cdot_\ig\bar\va_\vr-f(|\bar\va_\vr|)\bar\va_\vr=J_1+J_2+\cdots+J_6\in\ch^*
\]
where
\begin{eqnarray*}
&&J_1=\vr\ov{\nabla\eta\cdot_{\ig_{\R^m}}\psi_\vr}, \\
&&J_2=\eta \cdot \big(f(|\bar\psi_\vr|)-f(|\bar\va_\vr|)\big)\bar\psi_\vr, \\
&&J_3=\vr\eta W\cdot_\ig\bar\psi_\vr, \\
&&J_4=\vr \eta X\cdot_\ig\bar\psi_\vr, \\
&&J_5=\vr \eta \sum_{i,j}(b_{ij}-\de_{ij})\ov{\pa_i\cdot_{\ig_{\R^m}}\nabla_{\pa_j}\psi_\vr}, \\
&&J_6=\vr\sum_{i,j}(b_{ij}-\de_{ij})\ov{\pa_j\eta\pa_i\cdot_{\ig_{\R^m}}\psi_\vr}.
\end{eqnarray*}

In the following estimates we use that the support of $\eta$ is contained in $B_{2r}(0)\subset\R^m$. Using the exponential decay in Lemma \ref{exponential} and \eqref{metric development}-\eqref{WX}, we have:\\

$\displaystyle \aligned
\|J_1\|_{\ch\to\R}&\lesssim\bigg(\frac1{\vr^m}\int_{B_{2r}(y)}\big| \vr\ov{\nabla\eta\cdot_{\ig_{\R^m}}\psi_\vr} \big|^2d\vol_\ig \bigg)^{\frac12}
\lesssim\bigg(\frac1{\vr^m}\int_{r\leq|x|\leq2r}\big| \vr\psi_\vr \big|^2 dx \bigg)^{\frac12}\\
&\lesssim\vr\bigg(\int_{\frac{r}\vr\leq|x|\leq\frac{2r}{\vr}}\big|\psi \big|^2dx \bigg)^{\frac12}
\lesssim \vr\exp\Big(-\frac{c\cdot r}{\vr} \Big),
\endaligned$\\

$\displaystyle \aligned
\|J_2\|_{\ch\to\R}&\lesssim\bigg(\frac1{\vr^m}\int_{B_{2r}(y)\setminus B_r(y)}\big|\bar\psi_\vr \big|^2d\vol_\ig \bigg)^{\frac12}+\bigg(\frac1{\vr^m}\int_{B_{2r}(y)\setminus B_r(y)}\big|\bar\psi_\vr \big|^p d\vol_\ig \bigg)^{\frac{p-1}{p}} \\
&\lesssim\bigg(\int_{\frac{r}\vr\leq|x|\leq\frac{2r}{\vr}}|\psi|^2 dx \bigg)^{\frac12}+\bigg(\int_{\frac{r}\vr\leq|x|\leq\frac{2r}{\vr}}|\psi|^p dx \bigg)^{\frac{p-1}{p}}\lesssim\exp\Big(-\frac{c\cdot r}{\vr} \Big),
\endaligned$\\

$\displaystyle \aligned
\|J_3\|_{\ch\to\R}&\lesssim\bigg(\frac1{\vr^m}\int_{B_{2r}(y)}\big|\vr W\cdot_\ig\bar\psi_\vr\big|^2d\vol_\ig\bigg)^{\frac12}
\lesssim\bigg(\frac1{\vr^m}\int_{|x|\leq2r}\big(\vr |x|^3|\psi_\vr|\big)^2dx\bigg)^{\frac12} \\
&\lesssim \vr^4\bigg(\int_{|x|\leq\frac{2r}\vr} |x|^6|\psi|^2dx\bigg)^{\frac12} \lesssim \vr^4,
\endaligned$\\

$\displaystyle \aligned
\|J_4\|_{\ch\to\R}&\lesssim\bigg(\frac1{\vr^m}\int_{B_{2r}(y)}\big|\vr X\cdot_\ig\bar\psi_\vr\big|^2d\vol_\ig\bigg)^{\frac12}
\lesssim\bigg(\frac1{\vr^m}\int_{|x|\leq2r}\big(\vr |x||\psi_\vr|\big)^2dx\bigg)^{\frac12} \\
&\lesssim \vr^2\bigg(\int_{|x|\leq\frac{2r}\vr} |x|^2|\psi|^2dx\bigg)^{\frac12} \lesssim \vr^2,
\endaligned$\\

$\displaystyle \aligned
\|J_5\|_{\ch\to\R}&\lesssim\bigg(\frac1{\vr^m}\int_{|x|\leq2r}\big( \vr|x|^2|\nabla\psi_\vr| \big)^2 dx\bigg)^{\frac12}
\lesssim\vr^2\bigg(\int_{|x|\leq\frac{2r}\vr} |x|^4|\nabla\psi|^2 dx\bigg)^{\frac12}\lesssim \vr^2,
\endaligned$\\

$\displaystyle \aligned
\|J_6\|_{\ch\to\R}&\lesssim\bigg(\frac1{\vr^m}\int_{|x|\leq2r}\big( \vr|x|^2|\psi_\vr| \big)^2 dx\bigg)^{\frac12}
\lesssim\vr^3\bigg( \int_{|x|\leq\frac{2r}\vr} |x|^4|\psi|^2 dx \bigg)^{\frac12}\lesssim\vr^3.
\endaligned$\\

\noindent
Here, in the estimates of $J_1$ and $J_2$, the constant $c>0$ comes from the exponential decay in Lemma \ref{exponential}.
From all these, we finally obtain $\|\cl_\vr'(\bar\va_\vr)\|_\vr\leq O(\vr^2)$.
\end{proof}

We define a functional $\Theta: M\times\cb\to\R$ by
\[
\aligned
\Theta(y,\psi)&=\frac16\int_{\R^m}\text{Ric}_y(x,x)\Big(\frac12f(|\psi|)|\psi|^2-F(|\psi|) \Big)dx \\
&\qquad +\frac1{12}\sum_{i,j}\Real\int_{\R^m}\rr_y(e_i,x,x,e_j)
(\nabla_{\pa_j}\psi,\pa_i\cdot_{\ig_{\R^m}}\psi)dx
\endaligned
\]
Then we have the following

\begin{Lem}\label{estimate for L}
Let $\mu_0$ be the ground state energy of Eq. \eqref{limit equ} defined in \eqref{C0} and $\bar\va_\vr$ be defined in \eqref{spinor Rm}, then
\[
\cl_\vr(\bar\va_\vr)=\mu_0-\vr^2\Theta(y,\psi)+o(\vr^2)
\]
as $\vr\to0$.
\end{Lem}
\begin{proof}
By \eqref{cut-off spinor identity} and \eqref{X2} again, we have
\[
\frac1{\vr^m}\int_M(\vr\bar D\bar\va_\vr,\bar\va_\vr) d\vol_\ig +\frac{a}{\vr^m}\int_M(\om_\C\cdot_\ig\bar\va_\vr,\bar\va_\vr)d\vol_\ig =I_1+I_2+\dots+I_7,
\]
where
\begin{eqnarray*}
&&I_1=\frac{\Real}{\vr^m}\int_M\eta\cdot(\vr\ov{\nabla\eta\cdot_{\ig_{\R^m}}\psi_\vr},\bar\psi_\vr)d\vol_\ig, \\
&&I_2=\frac{1}{\vr^m}\int_Mf(|\bar\va_\vr|)|\bar\va_\vr|^2d\vol_\ig, \\
&&I_3=\frac{1}{\vr^m}\int_M\eta^2\big(f(|\bar\psi_\vr|)-f(|\bar\va_\vr|)\big)|\bar\psi_\vr|^2d\vol_\ig, \\
&&I_4=\frac{\Real}{\vr^m}\int_M\eta^2\cdot(\vr W\cdot_\ig\bar\psi_\vr,\bar\psi_\vr)d\vol_\ig, \\
&&I_5=\frac{\Real}{\vr^m}\int_M\eta^2\cdot(\vr X\cdot_\ig\bar\psi_\vr,\bar\psi_\vr)d\vol_\ig, \\
&&I_6=\sum_{i,j}\frac{\Real}{\vr^m}\int_M\eta^2(b_{ij}-\de_{ij})(\vr\ov{\pa_i\cdot_{\ig_{\R^m}}\nabla_{\pa_j}\psi_\vr},\bar\psi_\vr)d\vol_\ig, \\
&&I_7=\sum_{i,j}\frac{\Real}{\vr^m}\int_M\eta\pa_j\eta\cdot(b_{ij}-\de_{ij})(\vr\ov{\pa_i\cdot_{\ig_{\R^m}}\psi_\vr},\bar\psi_\vr)d\vol_\ig.
\end{eqnarray*}
Analogously to the arguments in Lemma \ref{estimate for L-derivative}, we shall use \eqref{metric development}-\eqref{WX} and the fact $(\ov{\nabla\eta\cdot_{\ig_{\R^m}}\psi_\vr},\bar\psi_\vr)\in i\R$ to obtain

$\displaystyle \aligned
I_1=0,
\endaligned$\\

$\displaystyle \aligned
I_2&=\frac1{\vr^m}\int_{|x|\leq r}f(|\psi_\vr|)|\psi_\vr|^2dx-\frac1{6\,\vr^m}\int_{|x|\leq r}f(|\psi_\vr|)|\psi_\vr|^2\text{Ric}_y(x,x)dx+o(\vr^2) \\
&=\int_{\R^m}f(|\psi|)|\psi|^2dx-\frac{\vr^2}{6}\int_{\R^m}f(|\psi|)|\psi|^2\text{Ric}_y(x,x)dx+o(\vr^2), 
\endaligned$\\

\medskip

$\displaystyle \aligned
|I_3|&\lesssim\frac1{\vr^m}\int_{r\leq|x|\leq2r}|\psi_\vr|^2dx+\frac1{\vr^m}\int_{r\leq|x|\leq2r}|\psi_\vr|^pdx\\
&\lesssim\int_{\frac{r}\vr\leq|x|\leq\frac{2r}\vr}|\psi|^2dx+\int_{\frac{r}\vr\leq|x|\leq\frac{2r}\vr}|\psi|^pdx\lesssim o(\vr^2),
\endaligned$\\

\medskip

$\displaystyle \aligned
|I_4|\lesssim\frac1{\vr^m}\int_{|x|\leq2r}\vr |x|^3|\psi_\vr|^2dx
\lesssim \vr^4\int_{|x|\leq\frac{2r}\vr}|x|^3|\psi|^2dx\lesssim \vr^4,
\endaligned$\\

$\displaystyle \aligned
I_5=0,
\endaligned$\\

$\displaystyle \aligned
I_6&=\sum_{i,j}\frac{\Real}{6\,\vr^m}\int_{|x|\leq r}\rr_y(e_i,x,x,e_j)(\vr\pa_i\cdot_{\ig_{\R^m}}\nabla_{\pa_j}\psi_\vr,\psi_\vr)dx+o(\vr^2)\\
&=\frac{\vr^2}{6}\sum_{i,j}\Real\int_{\R^m}\rr_y(e_i,x,x,e_j)(\pa_i\cdot_{\ig_{\R^m}}\nabla_{\pa_j}\psi,\psi)dx+o(\vr^2)\\
&=-\frac{\vr^2}{6}\sum_{i,j}\Real\int_{\R^m}\rr_y(e_i,x,x,e_j)(\nabla_{\pa_j}\psi,\pa_i\cdot_{\ig_{\R^m}}\psi)dx+o(\vr^2),
\endaligned$\\

$\displaystyle \aligned
I_7=0.
\endaligned$\\

\noindent
Combining all these estimates and the fact $\psi\in\cb$, we deduce that
\[
\aligned
\cl_\vr(\bar\va_\vr)&=\frac{I_1+I_2+\cdots+I_7}2-\frac1{\vr^m}\int_MF(|\bar\va_\vr|)d\vol_\ig\\
&=\frac12\int_{\R^m}f(|\psi|)|\psi|^2dx-\int_{\R^m}F(|\psi|)dx-\vr^2\Theta(y,\psi)+o(\vr^2) \\[0.5em]
&=\mu_0-\vr^2\Theta(y,\psi)+o(\vr^2)
\endaligned
\]
as desired.
\end{proof}

\subsection{Characterization of the concentration profile}

From Proposition \ref{reduction1} and \eqref{ga-vr}, we deduce that
\[
\mu_\vr=\inf_{u\in\ch_\vr^+\setminus\{0\}}
\max_{\psi\in \R u\op \ch_\vr^-}\cl_\vr(\psi)
=\inf_{u\in\ch_\vr^+\setminus\{0\}}\max_{t>0} I_\vr(tu)
=\inf_{u\in\msn_\vr}I_\vr(u)
\]
is a critical value. As in Proposition \ref{reduction1}, we take $\bar\psi_\vr$ to be the corresponding critical point of $\cl_\vr$. Then, as it was mentioned in \eqref{eeee2}, we find
\begin{\equ}\label{xx1}
\frac1{\vr^m}\int_M \frac12 f(|\bar\psi_\vr|)|\bar\psi_\vr|^2-F(|\bar\psi_\vr|)d\vol_\ig=\mu_\vr\geq\tau_0
\end{\equ}
for some $\tau_0>0$.
In what follows, for any $\xi\in M$ and $r>0$, $B_r(\xi)\subset M$ denotes the
ball of radius $r$ with respect to the metric $\ig$.

\begin{Lem}\label{step1}
There exist $y_\vr\in M$, $r_0,\de_0>0$ such that
\[
\liminf_{\vr\to0}\frac1{\vr^m}\int_{B_{\vr r_0}(y_\vr)}|\bar\psi_\vr|^2d\vol_\ig \geq \de_0.
\]
\end{Lem}
\begin{proof}
Assume to the contrary that for any $r>0$
\begin{\equ}\label{a2}
\sup_{\xi\in M}\frac1{\vr^m}\int_{B_{2\vr r}(\xi)}|\bar\psi_\vr|^2 d\vol_\ig \to0
\quad \text{as } \vr\to0.
\end{\equ}
For each $\xi\in M$, we choose a smooth real cut-off function
$\bt_{\xi,\vr}\equiv1$ on $B_{\vr r}(\xi)$ and $\supp\bt_{\xi,\vr}\subset B_{2\vr r}(\xi)$.
Then, for $s\in(0,1)$, we consider $q_s=2+(m^*-2)s\in(2,m^*)$ and we have
\[
\int_{B_{2\vr r}(\xi)}|\bt_{\xi,\vr}\bar\psi_\vr|^{q_s}d\vol_\ig\leq
\Big( \int_{B_{2\vr r}(\xi)}|\bt_{\xi,\vr}\bar\psi_\vr|^2 d\vol_\ig\Big)^{1-s}
\Big( \int_{B_{2\vr r}(\xi)}|\bt_{\xi,\vr}\bar\psi_\vr|^{\frac{2m}{m-1}} d\vol_\ig\Big)^s.
\]
Taking $s=\frac2{m^*}$, we obtain from Lemma \ref{embeding lemma} that
\[
\Big( \frac1{\vr^m}\int_{B_{2\vr r}(\xi)}|\bt_{\xi,\vr}\bar\psi_\vr|^{m^*} d\vol_\ig\Big)^s
\leq C \|\bt_{\xi,\vr}\bar\psi_\vr\|_\vr^2.
\]
We now cover $M$ by balls of radius $\vr r$ such that any point $\xi\in M$ is contained
in at most $K_M$ balls, where $K_M$ does not depend on $\vr$. In fact, one may take $K_M = 1+dim M$. Consequently we have
\[
\frac1{\vr^m}\int_M|\bar\psi_\vr|^{q_s}d\vol_\ig\leq C\cdot K_M\Big(
\sup_{\xi\in M}\int_{B_{2\vr r}(\xi)}|\bt_{\xi,\vr}\bar\psi_\vr|^2 d\vol_\ig\Big)^{1-s}
\|\bar\psi_\vr\|_\vr^2.
\]
Since $\|\bar\psi_\vr\|_\vr$ is bounded, it follows from \eqref{a2} that
$|\bar\psi_\vr|_{q_s,\vr}\to0$, and since $2<q_s<m^*$, we see easily that
$|\bar\psi_\vr|_{q,\vr}\to0$ for all $q\in(2,m^*)$ which contradicts \eqref{xx1}
\end{proof}

\begin{Lem}\label{step2}
$\lim_{\vr\to0}\mu_\vr=\mu_0$.
\end{Lem}
\begin{proof}
By the estimates obtained in Corollary \ref{key corollary}, Lemma \ref{estimate for L-derivative} and \ref{estimate for L}, we only need to show that
\[
\liminf_{\vr\to0}\mu_\vr\geq\mu_0.
\]

For this purpose, let us take a sequence $\{y_\vr\}\subset M$ and constants $r_0,\de_0>0$ such that Lemma \ref{step1} is valid. Given $0<r<\textit{inj}(M)/2$ arbitrarily, let $\bt_\vr\in C^\infty(M,[0,1])$ be such that $\bt_\vr\equiv1$ on $B_r(y_\vr)$ and $\supp \bt_\vr\subset B_{2r}(y_\vr)$. Via the Bourguignon-Gauduchon trivialization between the spinor bundles $\mbs(B_r(y_\vr))\to \mbs(B_r(0))$ and the rescaling $x\mapsto \frac{x}\vr$ on $\R^m$, the spinor field $\bt_\vr \bar\psi_\vr$ corresponds to a spinor field $z_\vr(\cdot)$ on $B_{2r/\vr}(0)\subset\R^m$.

Since $\bt_\vr\bar\psi_\vr$ is bounded in $\ch$ with respect to the norm $\|\cdot\|_\vr$, one sees easily that $z_\vr$ is bounded in $W^{\frac12,2}(B_R(0),\tilde\mbs(B_R(0)))$ for any $R>0$ as $\vr\to0$. By the fact
\[
\int_{|x|\leq \frac{r}\vr}|z_\vr|^{m^*}dx\lesssim\frac1{\vr^m}\int_{B_r(y_\vr)}|\bar\psi_\vr|^{m^*}d\vol_\ig
\lesssim \|\bar\psi_\vr\|_\vr<\infty
\]
for all small $\vr$, it follows that there exists $z_0\in L^{m^*}(\R^m,\tilde\mbs(\R^m))\cap W^{\frac12,2}_{loc}(\R^m,\tilde\mbs(\R^m))$  such that $z_\vr\rightharpoonup z_0$ in $W^{\frac12,2}_{loc}(\R^m,\tilde\mbs(\R^m))$ weakly and $z_\vr\to z_0$ in $L_{loc}^q(\R^m,\tilde\mbs(\R^m))$ for $2\leq q<\frac{2m}{m-1}$.

Let $\va\in W^{\frac12,2}(\R^m,\tilde\mbs(\R^m))$ be such that
$\supp \va$ is compact, i.e. $\supp \va\subset B_R^0$ for some $R>0$ large. Then, by a similar identity of \eqref{cut-off spinor identity} and \eqref{metric development}-\eqref{WX}, we have
\[
\aligned
&\int_{\R^m}\big( \tilde D_{\ig_{\R^m}}z_0 + a\om_\C\cdot_{\ig_{\R^m}}z_0
-f(|z_0|)z_0, \va \big)dx  \\
&\qquad =\lim_{\vr\to0}\int_{\supp\va}\big( \tilde D_{\ig_{\R^m}}z_\vr
+a\om_\C\cdot_{\ig_{\R^m}}z_\vr-f(|z_\vr|)z_\vr, \va\big)
d\vol_{\ig_{\Theta_\vr}}  \\
&\qquad =\lim_{\vr\to0}\frac1{\vr^m}\int_{B_{\vr R}(\xi_\vr)}
\big( \vr \tilde D_\ig \psi_\vr+ a\om_\C\cdot_\ig \psi_\vr
-f(|\psi_\vr|)\psi_\vr, \bar\va_\vr\big) d\vol_\ig\\
&\qquad =0
\endaligned
\]
where $\va_\vr(x)=\va(\frac{x}\vr)$ and $\bar\va_\vr$ is the spinor on $M$ defined via the Bourguignon-Gauduchon trivialization. Hence $z_0$ satisfies
\begin{\equ}\label{z-infty}
\tilde D_{\ig_{\R^m}}z_0 + a\om_\C\cdot_{\ig_{\R^m}}z_0
=f(|z_0|)z_0 \quad \text{on } \R^m.
\end{\equ}
As a consequence of $(f_1)$-$(f_3)$ and by elliptic regularity, we have $\tilde D_{\ig_{\R^m}}z_0 + a\om_\C\cdot_{\ig_{\R^m}}z_0\in L^{\frac{m^*}{m^*-1}}(\R^m,\tilde\mbs(\R^m))$.
Moreover, combined with the Sobolev embedding
$L^{\frac{p}{p-1}}(\R^m,\tilde\mbs(\R^m))\hookrightarrow
W^{-\frac12,2}(\R^m,\tilde\mbs(\R^m))$, we get $z_0\in
W^{\frac12,2}(\R^m,\tilde\mbs(\R^m))$.

Now, by Lemma \ref{step1}, one sees  that $z_0$ is a non-trivial solution to \eqref{z-infty}. In virtue of \eqref{xx1}, we conclude that
\[
\aligned
\liminf_{\vr\to0}\cl_\vr(\bar\psi_\vr)&=\liminf_{\vr\to0}
\frac1{\vr^m}\int_M \frac12f(|\bar\psi_\vr|)|\bar\psi_\vr|^2-F(|\bar\psi_\vr|)d\vol_\ig  \\
&\geq\liminf_{\vr\to0}\int_{B_R(0)}\frac12f(|z_\vr|)|z_\vr|^2-F(|z_\vr|)dx\\
&\geq\int_{B_R(0)}\frac12f(|z_0|)|z_0|^2-F(|z_0|)dx
\endaligned
\]
where in the last inequality we have used the Fatou's lemma.
Due to the arbitrariness of $R>0$, together with \eqref{C0},
we obtain
\[
\liminf_{\vr\to0}\mu_\vr=\liminf_{\vr\to0}\cl_\vr(\bar\psi_\vr)\geq \mu_0.
\]
\end{proof}

Now, we fix the sequence $\{y_\vr\}\subset M$ and constants $r_0,\de_0>0$ as in Lemma \ref{step1}. Up to a subsequence, we also assume that $y_\vr\to y_0$ in $M$ as $\vr\to0$. As Lemma \ref{step2} alluded, there exists a least-energy solution $z_0$ of the limit equation \eqref{limit equ} corresponding to the weak limit of the solutions $\bar\psi_\vr$ via  Bourguignon-Gauduchon trivialization and rescaling. Without loss of generality, one may assume $y_\vr$ to be the maximum point of $|\bar\psi_\vr|$, and then $z_0\in\cb$.

Choose $\bt\in C^\infty(M,[0,1])$ be such that $\bt\equiv1$ on $B_r(y_0)$ and $\supp \bt\subset B_{2r}(y_0)$ for some $r<\textit{inj}(M)/2$. Set
$\bar\phi_\vr=\bt\bar z_{0,\vr}$, where $z_{0,\vr}(x)=z_0(x/\vr)$ and $\bar z_{0,\vr}$ is the corresponding spinor field obtained via Bourguignon-Gauduchon trivialization. We have the following asymptotic characterization.

\begin{Lem}\label{step3}
$\|\bar\psi_\vr-\bar\phi_\vr\|_\vr\to0$ as $\vr\to0$.
\end{Lem}
\begin{proof}
Setting $w_\vr=\bar\psi_\vr-\bar\phi_\vr$, we have $|w_\vr|_{q,\vr}\to0$ as
$\vr\to0$ for all $q\in(2,m^*)$ (otherwise, we can apply Lemma \ref{step1} and \ref{step2} for $\{w_\vr\}$ instead
of $\{\bar\psi_\vr\}$ to get $\cl_\vr(\bar\psi_\vr)\geq 2\mu_0$ which is absurd).

To proceed, let us go over the proof of Lemma \ref{step2} again to see that
\[
\mu_0=\lim_{\vr\to0}\frac1{\vr^m}\int_M\frac12f(|\bar\psi_\vr|)|\bar\psi_\vr|^2-F(|\bar\psi_\vr|)d\vol_\ig
=\int_{\R^m}\frac12f(|z_0|)|z_0|^2-F(|z_0|)dx
\]
Similar to \eqref{X0} (but here we need to transplant it on the manifold), it is easy to see that for every $\epsilon>0$ there is $R>0$ large such that
\begin{\equ}\label{X3}
\limsup_{\vr\to0}\int_{M\setminus B_{\vr R}(y_\vr)}\frac12f(|\bar\psi_\vr|)|\bar\psi_\vr|^2-F(|\bar\psi_\vr|)d\vol_\ig\leq\epsilon.
\end{\equ}
From the fact $\cl_\vr'(\bar\psi_\vr)=0$ and the estimate in Lemma \ref{estimate for L-derivative}, we have
\[
\inp{\bar\psi_\vr^+}{w_\vr^+}_\vr+\inp{\bar\psi_\vr^-}{w_\vr^-}_\vr
-\frac{\Real}{\vr^m}\int_M f(|\bar\psi_\vr|)(\bar\psi_\vr,w_\vr^+-w_\vr^-)d\vol_\ig=0
\]
and
\[
\inp{\bar\phi_\vr^+}{w_\vr^+}_\vr+\inp{\bar\phi_\vr^-}{w_\vr^-}_\vr
-\frac{\Real}{\vr^m}\int_M f(|\bar\phi_\vr|)(\bar\phi_\vr,w_\vr^+-w_\vr^-)d\vol_\ig=o_\vr(1).
\]
Hence, we get
\begin{\equ}\label{X4}
\|w_\vr\|_\vr^2=\frac{\Real}{\vr^m}\int_M f(|\bar\psi_\vr|)(\bar\psi_\vr,w_\vr^+-w_\vr^-)d\vol_\ig+o_\vr(1),
\end{\equ}
where we have used that $|w_\vr|_{q,\vr}\to0$ as
$\vr\to0$ for all $q\in(2,m^*)$. Recall that, by $(f_1)$-$(f_3)$, for arbitrarily small $\de>0$ there exists $c_\de>0$ such that $f(s)\leq \de+c_\de(f(s)s^2)^{\frac{p-1}{p}}$ and $F(s)\leq\frac1{\theta+2}f(s)s^2$ for all $s\geq0$. Thus, it follows from \eqref{X3}, \eqref{X4} and the Sobolev embeddings that
\[
\aligned
\|w_\vr\|_\vr^2
\leq\de C\|w_\vr\|_\vr+ C_\de \big(\epsilon+o_\vr(1)\big)^{\frac{p-1}{p}}\|w_\vr\|_\vr+o_\vr(1),
\endaligned
\]
for some constants $C, C_\de>0$.
Due to the arbitrariness of $\de,\epsilon>0$, one sees $\|w_\vr\|_\vr\to0$ as $\vr\to\infty$.
\end{proof}

\begin{Lem}\label{exponentially decay}
For small $\vr>0$, there exist positive constants $C,c>0$ such that
\[
    |\bar\psi_\vr(\xi)|\leq C\exp\Big( -\frac{\,c\,}{\vr}\dist(\xi,y_\vr) \Big),    \quad \text{for all $\xi\in M$.}
\]
\end{Lem}
\begin{proof}
We first show that the family $\{\bar\psi_\vr\}$ decays uniformly on $M$ in the following sense: 
\begin{\equ}\label{decay1}
\frac{\dist(\xi,y_\vr)}{\vr}\to\infty \Longrightarrow|\bar\psi_\vr(\xi)|\to0, \quad \text{as } \vr\to0.
\end{\equ}

Indeed, since $\bar\psi_\vr$ solves \eqref{equ3} and $\cl_\vr(\bar\psi_\vr)=\mu_\vr\to\mu_0$, by the Sobolev embedding and bootstrap arguments, we deduce that $\{|\bar\psi_\vr|_\infty\}$ is bounded. Applying the operator $\vr\tilde D_\ig$ on both sides of \eqref{equ3}, and using the Lichnerowicz formula, we find
\[
\vr^2\De_\ig\bar\psi_\vr-\frac{\vr^2 Scal_\ig}4\bar\psi_\vr=(a^2-f(|\bar\psi_\vr|)^2)\bar\psi_\vr-\vr\nabla_\ig f(|\bar\psi_\vr|)\cdot_\ig\bar\psi_\vr \quad \text{on } M,
\]
where $\De_\ig=\div_\ig\nabla_\ig$ is the Laplace-Beltrami operator, $\nabla_\ig$ is the gradient and $Scal_\ig$ is the scalar curvature with respect to the metric $\ig$. Using the fact
\[
\De_\ig|\bar\psi_\vr|^2=2\Real(\De_\ig\bar\psi_\vr,\bar\psi_\vr)
+2|\nabla_\ig\bar\psi_\vr|^2,
\]
we have
\begin{\equ}\label{decay2}
\vr^2\De_\ig|\bar\psi_\vr|^2=2\Big( a^2+\frac{\vr^2 Scal_\ig}4- f(|\bar\psi_\vr|)^2\Big)|\bar\psi_\vr|^2+2|\nabla_\ig\bar\psi_\vr|^2\geq -C |\bar\psi_\vr|^2
\end{\equ}
for some $C>0$. Then, it follows from the local boundedness of sub-solutions (see for example \cite[Theorem 4.1]{HF}) that
\begin{\equ}\label{sub-solu esti}
|\bar\psi_\vr(\xi)|^2\leq \frac{C_0}{\vr^m}\int_{B_\vr(\xi)}|\bar\psi_\vr|^2d\vol_\ig
\end{\equ}
with $C_0>0$ independent of $\xi\in M$ and $\vr>0$.

Assume by contradiction that there exist $\de>0$ and $\xi_\vr\in M$ such that $\dist(\xi_\vr,y_\vr)/\vr\to\infty$ as $\vr\to0$ and $|\bar\psi_\vr(\xi_\vr)|\geq\de$. By \eqref{sub-solu esti} and Lemma \ref{step3}, we deduce, as $\vr\to0$
\[
\aligned
\de&\leq C_0\Big(\frac1{\vr^m}\int_{B_\vr(\xi_\vr)}|\bar\psi_\vr|^2d\vol_\ig\Big)^{\frac12} \leq C_0 |\bar\psi_\vr-\bar\phi_\vr|_{2,\vr}+C_0\Big(\frac1{\vr^m}\int_{B_\vr(\xi_\vr)}|\bar\phi_\vr|^2d\vol_\ig\Big)^{\frac12} \\
&\leq o_\vr(1)+C_0\Big(\int_{B_1\big(\frac{\exp_{y_\vr}^{-1}(\xi_\vr)}{\vr}\big)}|z_0|^2d\vol_\ig\Big)^{\frac12}\to0
\endaligned
\]
which is a contradiction; here $\bar\phi_\vr$ and $z_0$ are as in Lemma \ref{step3}. This proves \eqref{decay1}.

In order to see the exponential decay, by \eqref{decay1} and \eqref{decay2}, we can take $R_0>0$ sufficiently large so that
\begin{\equ}\label{decay3}
\vr^2\De_\ig|\bar\psi_\vr|^2\geq a^2 |\bar\psi_\vr|^2 \quad
\text{on } M\setminus B_{\vr R_0}(y_\vr)
\end{\equ}
for all $\vr>0$ small. Our next strategy is to use the comparison principle and is very similar to the flat case (see Lemma \ref{exponential}): we first construct a sequence of positive functions $\bar \Ga_\vr$ on $M\setminus B_{\vr R_0}(y_\vr)$ which decays exponentially,  then we chose $C>0$ such that $|\bar\psi_\vr(\xi)|^2\leq C\cdot \bar\Ga(\xi)$ for $\xi\in\partial B_{\vr R_0}(y_\vr)$ and finally, by setting $U_\vr=|\bar\psi_\vr|^2- C\cdot \bar\Ga$, we will show
\[
\vr^2 \De_\ig U_\vr \geq a^2 U_\vr, \quad \text{on } M\setminus B_{\vr R_0}(y_\vr).
\]

For the sake of clarity, let us consider the (local) normal coordinates $\exp_{y_\vr}^{-1}: B_{r}(y_\vr)\to \R^m\cong T_{y_\vr}M$ for a fixed $r<\textit{inj}(M)/5$ and set $v_\vr(x)=|\bar\psi_\vr|^2\circ\exp_{y_\vr}(\vr x)$. Then we find 
\[
\frac1{\sqrt{\text{det\,}\ig(\vr x)}}\sum_{j,k}\pa_j\big( \ig^{jk}(\vr x)\sqrt{\text{det\,}\ig(\vr x)}\pa_k v_\vr\big)-a^2 v_\vr\geq0
\quad \text{for } R_0\leq |x|\leq \frac{r}\vr,
\]
where $(\ig^{ij})_{ij}=G^{-1}$ is the inverse matrix of the metric $\ig$.

\smallskip

For each $\vr>0$ small, we take $\Ga_\vr(x)=e^{-\frac a2 L_\vr(|x|)}$, where $L_\vr:[0,\infty)\to[0,\infty)$ is a function defined as
\begin{\equ}\label{def L-vr}
L_\vr(\rho)=\left\{
\aligned
& \rho & \quad & 0\leq\rho<\frac r\vr ,\\[0.3em]
&\frac12\Big( \rho +\frac r\vr\sin\big(\frac\vr r\rho-1 \big)  \Big) +\frac r{2\vr} & \quad &
\frac r\vr\leq \rho < \frac{\pi+1}\vr r ,\\[0.3em]
&\frac{\pi+2}{2\vr}r & \quad & \rho\geq \frac{\pi+1}\vr r.
\endaligned
\right.
\end{\equ}
According to the above definition, $\Ga_\vr\in C^2(\R^m\setminus\{0\})$ is spherically symmetric, and so satisfies
\[
\Delta\Ga_\vr-\frac{a^2}4\Ga_\vr=-\frac a2 e^{-\frac a2L_\vr(\rho)} \Big( L_\vr''(\rho)-\frac a2 (L_\vr'(\rho))^2 + \frac{m-1}\rho L_\vr'(\rho) + \frac a2 \, \Big)
\quad \text{for } \rho=|x|>0.
\]
Now we see easily that
\[
\Delta\Ga_\vr-\frac{a^2}4\Ga_\vr=\left\{
\aligned
&-\frac{a(m-1)}{2\rho} e^{-\frac a2L_\vr(\rho)}
& \quad & 0\leq\rho<\frac r\vr ,\\[0.3em]
& -\frac {a^2}4 e^{-\frac a2L_\vr(\rho)} & \quad & \rho\geq \frac{\pi+1}\vr r.
\endaligned
\right.
\]
And for $\frac r\vr\leq \rho < \frac{\pi+1}\vr r$, from a direct computation, we have
\[
\aligned
&L_\vr''(\rho)-\frac a2 (L_\vr'(\rho))^2 + \frac{m-1}\rho L_\vr'(\rho) + \frac a2 \\
 &\qquad=
\frac a2-\frac a8 \Big( 1+\cos\big( \frac\vr r\rho-1 \big) \Big)^2-\frac\vr{2r}\sin\big( \frac\vr r\rho-1 \big)+\frac{m-1}{2\rho}\Big( 1+\cos\big( \frac\vr r\rho-1 \big) \Big) \\
&\qquad \geq
\frac a2-\frac a8 \Big( 1+\cos\big( \frac\vr r\rho-1 \big) \Big)^2-\frac\vr{2r}\sin\big( \frac\vr r\rho-1 \big)+\frac{\vr}{2(\pi+1)r}\Big( 1+\cos\big( \frac\vr r\rho-1 \big) \Big) \\
&\qquad =\frac a2-\frac a8 \Big( 1+\cos\big( \frac\vr r\rho-1 \big) \Big)^2 - \frac{\vr\sqrt{(\pi+1)^2+1}}{2(\pi+1)r}\sin\big( \frac\vr r\rho-1 -\vartheta\big)+\frac{\vr}{2(\pi+1)r}
\endaligned
\]
where $\vartheta=\arcsin\frac{1}{\sqrt{(\pi+1)^2+1}}\in(0,\frac\pi2)$. Therefore, for small $\vr>0$, we can derive that
\[
\Delta\Ga_\vr-\frac{a^2}4\Ga_\vr<0 \quad \text{on } \R^m\setminus\{0\}.
\]
Moreover, about the gradient of $\Ga_\vr$ we have 
\[
\frac{|\nabla\Ga_\vr(x)|}{\Ga_\vr(x)}\leq\frac a2
\quad \text{for all } |x|>0 \text{ and } \vr>0.
\]
Hence, by taking $r$ smaller if necessary, it follows from \eqref{metric development} that
\[
\aligned
\frac1{\sqrt{\text{det\,}\ig(\vr x)}}\sum_{j,k}\pa_j\big( \ig^{jk}(\vr x)\sqrt{\text{det\,}\ig(\vr x)}\pa_k \Ga_\vr\big)\leq\De\Ga_\vr+
o_r(1)\big(|\De\Ga_\vr|+|\nabla\Ga_\vr|\big)
\leq a^2\Ga_\vr
\endaligned
\]
for $R_0\leq |x|\leq {r}/\vr$.

Next, using \eqref{decay1}, we can fix $C_0>0$ such that $v_\vr(x)\leq  C_0\cdot\Ga_\vr(x)=C_0 e^{-\frac a2 R_0}$ for $|x|=R_0$ and $\vr$ small. Via the Riemannian normal coordinates, we can transplant the functions $\Ga_\vr$ onto $M\setminus B_{\vr R_0}(y_\vr)$ by introducing
\begin{\equ}\label{def bar Ga-vr}
\bar\Ga_\vr(\xi)=\left\{
\aligned
&\Ga_\vr\Big(\frac{\exp_{y_\vr}^{-1}(\xi)}{\vr} \Big)& \quad  & \xi\in B_{5r}(y_\vr)\setminus B_{\vr R_0}(y_\vr), \\[0.5em]
&e^{-\frac{a(\pi+2)}{4\vr}r } & \quad  & M\setminus B_{5r}(y_\vr).
\endaligned
\right.
\end{\equ}
Then, from \eqref{def L-vr}, we see that $\bar \Ga_\vr>0$ is a well-defined $C^2$-smooth function on $M\setminus B_{\vr R_0}(y_\vr)$, and in particular, 
\[
\vr^2\De_\ig\bar\Ga_\vr-a^2\bar\Ga_\vr\leq0
\quad \text{on } M\setminus B_{\vr R_0}(y_\vr).
\]
Setting $U_\vr=|\bar\psi_\vr|^2- C_0\cdot \bar\Ga_\vr(x)$, we can see from \eqref{decay3} that
\[
\vr^2\De_\ig U_\vr-a^2 U_\vr\geq0 \quad 
\text{on } M\setminus B_{\vr R_0}(y_\vr).
\]
And it is standard to show from the comparison principle that $U_\vr\leq 0$ on $M\setminus B_{\vr R_0}(y_\vr)$. 

Turning back to the definition \eqref{def bar Ga-vr}, we see that
\[
\bar\Ga_\vr(\xi)\leq \exp\Big(-\frac{c}\vr\dist(\xi,y_\vr)\Big)
\quad \text{for } \xi\in B_r(y_\vr).
\]
and
\[
\bar\Ga_\vr(\xi)\leq \exp\Big(-\frac{c\cdot r}\vr\Big)
\quad \text{for } \xi\in M\setminus B_r(y_\vr).
\]
Setting $d_M=\sup\{\dist(\xi_1,\xi_2):\, \xi_1,\xi_2\in M\}$, we see easily that $d_M\geq\textit{inj}(M)$ and so
\[
\bar\Ga_\vr(\xi)\leq \exp\Big(-\frac{c\cdot r}{\vr\cdot d_M}\dist(\xi,y_\vr)\Big)
\quad \text{for all } \xi\in M.
\]
This implies the exponential decay for $|\bar\psi_\vr|$ by simply taking the square root.
\end{proof}

Recall that in the proof of Lemma \ref{step2}, we have taken $\bt_\vr\in C^\infty(M,[0,1])$ be such that $\bt_\vr\equiv1$ on $B_r(y_\vr)$ and $\supp \bt_\vr\subset B_{2r}(y_\vr)$ for some $r<\textit{inj}(M)$. Via the Bourguignon-Gauduchon trivialization between the spinor bundles $\mbs(B_r(y_\vr))\to \mbs(B_r(0))$ and the rescaling $x\mapsto \frac{x}\vr$ on $\R^m$, the spinor field $\bt_\vr \bar\psi_\vr$ corresponds to a spinor field $z_\vr$ on $B_{2r/\vr}(0)\subset\R^m$. And, by Lemma \ref{step3} and bootstrap arguments, $z_\vr$ converges in $W^{1,q}(\R^m,\tilde\mbs(\R^m))$ to some $z_0\in\cb$ as $\vr\to0$, for $q\geq2$. Thanks to the fast decay rate of $\bar\psi_\vr$, in addition to Lemma \ref{step2}, we have the following refined lower bound estimate for the critical level $\mu_\vr$.

\begin{Lem}\label{estimate mu-vr lower}
Let $y_\vr$ be a maximum point of $|\bar\psi_\vr|$. Up to a subsequence if necessary, assume $y_\vr\to y_0\in M$ as $\vr\to0$ with respect to the Riemannian metric $\ig$. Then
\[
\mu_\vr\geq \mu_0-\vr^2\Theta(y_0,z_0)+o(\vr^2).
\]
\end{Lem}
\begin{proof}
Notice that $\cl_\vr'(\bar\psi_\vr)=0$ and
\[
\aligned
\cl_\vr(\bar\psi_\vr)-\cl_\vr(\bt_\vr\bar\psi_\vr)&=
\frac1{\vr^m}\int_M\frac12f(|\bar\psi_\vr|)|\bar\psi_\vr|^2-F(|\bar\psi_\vr|)d\vol_\ig \\
&\qquad -\frac1{\vr^m}\int_M\frac{\bt_\vr^2}2f(|\bar\psi_\vr|)|\bar\psi_\vr|^2-F(|\bt_\vr\bar\psi_\vr|)d\vol_\ig.
\endaligned
\]
By $(f_1)$-$(f_2)$, for each fixed $s\geq0$, we deduce that the function $t\mapsto \frac{t^2}2f(s)s^2-F(ts)$ is non-decreasing for $t\in[0,1]$. Hence, by $\bt_\vr(\xi)\in[0,1]$ for all $\xi\in M$, one sees easily
\[
\cl_\vr(\bt_\vr\bar\psi_\vr)\leq\cl_\vr(\bar\psi_\vr)=\mu_\vr.
\]

On the other hand, since $\bt_\vr \bar\psi_\vr$ corresponds to a spinor field $z_\vr$ on $B_{2r/\vr}(0)\subset\R^m$ through the Bourguignon-Gauduchon trivialization and rescaling, by developing the relationship in \eqref{cut-off spinor identity} for $\bt_\vr \bar\psi_\vr$ and $z_\vr$, we obtain the following correspondence of spinors
\[
\vr \bar D(\bt_\vr\bar\psi_\vr)\longleftrightarrow\ov{D z_\vr}+\vr^3 \ov{W\cdot_{\ig_{\R^m}}z_\vr}+\vr \ov{X\cdot_{\ig_{\R^m}}z_\vr}
+\vr^2\sum_{i,j}(b_{ij}-\de_{ij})\ov{\pa_i\cdot_{\ig_{\R^m}}\nabla_{\pa_j}z_\vr}.
\]
Now, we can argue similarly to the proof of Lemma \ref{estimate for L-derivative} and \ref{estimate for L} to obtain $\Phi'(z_\vr)=o(
\vr)$ and
\[
\aligned
\cl_\vr(\bt_\vr\bar\psi_\vr)&=\frac1{\vr^m}\int_M\frac12(\vr\bar D(\bt_\vr\bar\psi_\vr),\bt_\vr\bar\psi_\vr)+\frac{a}2(\om_\C\cdot_\ig\bt_\vr\bar\psi_\vr,\bt_\vr\bar\psi_\vr)
-F(|\bt_\vr\bar\psi_\vr|) d\vol_\ig \\
&=\int_{\R^m}\frac12( Dz_\vr,z_\vr)+\frac{a}2(\om_\C\cdot_{\ig_{\R^m}}z_\vr,z_\vr)
-F(|z_\vr|) dx  \\
&\qquad -\frac{\vr^2}{12}\int_{\R^m}\text{Ric}_{y_\vr}(x,x)(Dz_\vr,z_\vr)dx
-\frac{a\,\vr^2}{12}\int_{\R^m}\text{Ric}_{y_\vr}(x,x)(\om_\C\cdot_{\ig_{\R^m}}z_\vr,z_\vr)dx\\
&\qquad
+\frac{\vr^2}6\int_{\R^m}\text{Ric}_{y_\vr}(x,x)F(|z_\vr|) dx\\
&\qquad
-\frac{\vr^2}{12}\sum_{i,j}\Real\int_{\R^m}\rr_{y_\vr}(e_i,x,x,e_j)(\nabla_{\pa_j}z_\vr,\pa_i\cdot_{\ig_{\R^m}}z_\vr)dx
+o(\vr^2)
\endaligned
\]
Since $y_\vr\to y_0$ and  $z_\vr\to z_0$ in $W^{1,q}(\R^m,\tilde\mbs(\R^m))$ as $\vr\to0$, for $q\geq2$, and since $|z_\vr|$ decays exponentially, we have
\[
\aligned
&\int_{\R^m}\text{Ric}_{y_\vr}(x,x)\big[(Dz_\vr,z_\vr)+(a\om_\C\cdot_{\ig_{\R^m}}z_\vr,z_\vr)-2F(|z_\vr|) \big] dx \\
&\qquad =\int_{\R^m}\text{Ric}_{y_0}(x,x)\big[(Dz_0,z_0)+(a\om_\C\cdot_{\ig_{\R^m}}z_0,z_0)-2F(|z_0|) \big] dx+o_\vr(1)\\
&\qquad
=\int_{\R^m}\text{Ric}_{y_0}(x,x)\big[f(|z_0|)|z_0|^2-2F(|z_0|) \big] dx+o_\vr(1)
\endaligned
\]
and
\[
\aligned
&\sum_{i,j}\Real\int_{\R^m}\rr_{y_\vr}(e_i,x,x,e_j)(\nabla_{\pa_j}z_\vr,\pa_i\cdot_{\ig_{\R^m}}z_\vr)dx\\
&\qquad
=\sum_{i,j}\Real\int_{\R^m}\rr_{y_0}(e_i,x,x,e_j)(\nabla_{\pa_j}z_0,\pa_i\cdot_{\ig_{\R^m}}z_0)dx+o_\vr(1)
\endaligned
\]
Note that, by Lemma \ref{reduction limit equ} $(4)$, we also have
\[
\mu_0=\inf_{u\in \ce^+\setminus\{0\}}\max_{t>0}J(tu)\leq\Phi(z_\vr)+O(\|\Phi'(z_\vr)\|^2).
\]
Hence, it follows directly that
\[
\mu_\vr\geq \cl_\vr(\bt_\vr\bar\psi_\vr)\geq\mu_0-\vr^2\Theta(y_0,z_0)+o(\vr^2).
\]

\end{proof}

\begin{proof}[Proof of the main theorem]
We first see from Corollary \ref{key corollary}, Lemma \ref{estimate for L-derivative} and \ref{estimate for L} that
\[
\mu_\vr=\cl_\vr(\bar\psi_\vr)\leq \mu_0-\vr^2\max_{(y,\psi)\in M\times\cb}\Theta(y,\psi)+o(\vr^2) 
\]
for small $\vr>0$. Hence, by Lemma \ref{estimate mu-vr lower} and taking the limit $\vr\to0$, we have
\[
\Theta(y_0,z_0)\geq\max_{(y,\psi)\in M\times\cb}\Theta(y,\psi). 
\]
Therefore, we conclude that $(y_\vr,z_\vr)\to (y_0,z_0)$ in $M\times\cb$ such that
\[
\lim_{\vr\to0}\Theta(y_\vr,z_\vr)=\max_{(y,\psi)\in M\times\cb}\Theta(y,\psi), 
\]
and
\[
\cl_\vr(\bar\psi_\vr)=\mu_0-\vr^2\max_{(y,\psi)\in M\times\cb}\Theta(y,\psi)+o(\vr^2). 
\]

In the $2$-dimensional case, the Ricci tensor determines the whole curvature. Specifically, we have
\[
\text{Ric}(x,x)=\rr(e_1,x,x,e_1)+\rr(e_2,x,x,e_2)=\rr(e_1,e_2,e_2,e_1)|x|^2
\]
and
\[
Scal_\ig=\sum_{j=1}^2\sum_{i=1}^2\rr(e_j,e_i,e_i,e_j)=2\rr(e_1,e_2,e_2,e_1).
\]
Hence, by noting that the scalar curvature is twice the Gaussian curvature for surfaces, we have
\[
\aligned
\Theta(y,\psi)&=\frac{K_\ig(y)}{6}\int_{\R^2}\Big(\frac12f(|\psi|)|\psi|^2-F(|\psi|)\Big)|x|^2dx \\
&\qquad +\frac{K_\ig(y)}{12}\Real\int_{\R^2}\big( (x_2\nabla_{\pa_1}-x_1\nabla_{\pa_2})\psi, (x_2\pa_1-x_1\pa_2)\cdot_{\ig_{\R^2}}\psi \big) dx,
\endaligned
\]
where $K_\ig$ denotes the Gaussian curvature of $(M,\ig)$.

Finally, by substituting $f(|\psi|)|\psi|^2=|\psi|^{n^*}$ and $F(|\psi|)=\frac1{n^*}|\psi|^{n^*}$ into the above formulas, one completes the proof.
\end{proof}

\textit{\textbf{Acknowledgements}}. The authors wish to express their gratitude to the anonymous reviewer for his/her positive comments and constructive suggestions.

\vspace{2mm}
{\sc Thomas Bartsch\\
 Mathematisches Institut, Universit\"at Giessen\\
 35392 Giessen, Germany}\\
 Thomas.Bartsch@math.uni-giessen.de\\

{\sc Tian Xu\\
 Center for Applied Mathematics, Tianjin University\\
 Tianjin, 300072, China}\\
 xutian@amss.ac.cn

\end{document}